\documentclass[oneside,11pt]{article}

\usepackage{amsmath}
\usepackage{amssymb}
\usepackage{amsthm}
\usepackage[alphabetic]{amsrefs}
\usepackage{array}
\usepackage[all]{xy}
\usepackage{fancyhdr}
\usepackage{euscript}
\usepackage{graphics,graphicx}
\usepackage{cancel}
\usepackage{fancybox}
\usepackage{verbatim}
\usepackage{tikz}
\usepackage{longtable}
\usepackage[tableposition=below]{caption}
\usepackage[colorlinks=false,hyperindex]{hyperref}

\usepackage{chngcntr}
\counterwithin{equation}{section}

\hypersetup{
	colorlinks,
	citecolor=violet,
	linkcolor=violet,
	urlcolor=blue}

\DeclareFontFamily{OT1}{rsfs}{}
\DeclareFontShape{OT1}{rsfs}{n}{it}{<-> rsfs10}{}
\DeclareMathAlphabet{\mathscr}{OT1}{rsfs}{n}{it}

\newtheoremstyle%
{custom}%
{}
{}
{}
{}
{}
{.}
{ }
{\thmname{}
\thmnumber{}%
\thmnote{\bfseries #3}}%

\newtheoremstyle%
{Theorem}%
{}%
{}%
{\itshape}%
{}%
{}%
{.}%
{ }%
{\thmname{\bfseries #1}%
\thmnumber{\;\bfseries #2}%
\thmnote{\;(\bfseries #3)}}%

\theoremstyle{Theorem}
\newtheorem{theorem}{Theorem}[section]

\newtheorem{corollary}[theorem]{Corollary}
\newtheorem{lemma}[theorem]{Lemma}
\newtheorem{proposition}[theorem]{Proposition}
\theoremstyle{definition}
\newtheorem{definition}[theorem]{Definition}
\newtheorem{example}[theorem]{Example}
\newtheorem{remark}[theorem]{Remark}
\newtheorem{problem}[theorem]{Problem}
\newtheorem{conjecture}[theorem]{Conjecture}
\newtheorem{question}{Question}

\newcommand{\A}{\mathcal{A}}

\newcommand{\QQ}{\mathbf{Q}}
\newcommand{\CC}{\mathbf{C}}
\newcommand{\FF}{\mathbf{F}}
\newcommand{\Fq}{\mathbf{F}_q}
\newcommand{\ZZ}{\mathbf{Z}}
\newcommand{\Spec}{ \mathrm{Spec}}

\newcommand{\PP}{\mathbf{P}}

\newcommand{\Tr}{\mathrm{Tr}}

\newcommand{\rk}{\mathrm{rk}}
\newcommand{\RR}{\mathbf{R}}

\newcommand{\ord}{\mathrm{ord}}

\DeclareMathOperator{\Res}{Res}
\DeclareMathOperator{\inv}{inv}
\DeclareMathOperator{\lcd}{lcd}
\DeclareMathOperator{\Span}{Span}
\DeclareMathOperator{\Br}{Br}

\newcommand{\End}{\operatorname{End}}
\newcommand{\NN}{\mathbf{N}}
\newcommand{\alg}{\operatorname{alg}}
\newcommand{\OO}{\mathcal{O}}

\newcommand{\ind}{\operatorname{index}}
\newcommand{\Disc}{\operatorname{Disc}}
\newcommand{\gal}{\operatorname{gal}}
\newcommand{\maxp}{\operatorname{maxp}}

\DeclareMathOperator{\Gal}{Gal}
\DeclareMathOperator{\rd}{rd}

\newcommand{\avlink}[1]{\href{http://www.lmfdb.org/Variety/Abelian/Fq/#1}{\textsf{#1}}}
\newcommand{\nflink}[1]{\href{http://www.lmfdb.org/NumberField/#1}{\textsf{#1}}}

\begin{document}

\title{Isogeny Classes of Abelian Varieties over Finite Fields in the LMFDB}
\author{Taylor Dupuy, Kiran Kedlaya, David Roe, Christelle Vincent}

\maketitle

\begin{abstract}
This document is intended to summarize the theory and methods behind \verb+fq_isog+ collection inside the \verb+ab_var+ database in the LMFDB as well as some observations gleaned from these databases. 
This collection consists of tables of Weil $q$-polynomials, which by the Honda-Tate theorem are in bijection with isogeny classes of abelian varieties over finite fields. 
\end{abstract}

\tableofcontents

\section{Introduction}

The LMFDB (L-Functions and Modular Forms Database) includes a database of isogeny classes of abelian varieties defined over finite fields. 
This database can be accessed at \url{https://www.lmfdb.org/Variety/Abelian/Fq/}.
The purposes of this paper are on the one hand to document the theoretical results on which these methods depend (\S\ref{S:background}) and some of the algorithms used to compile the data (\S\ref{S:algorithms}); and on the other hand to extract some observations from the compiled data, including comparison of some statistical data with relevant heuristics
(\S\ref{S:statistics}) and collecting some examples pertaining to theorems and conjectures in the literature
(\S\ref{S:examples}).

For small dimension $g$ and prime power $q=p^r$, the database contains searchable lists of complete sets of isogeny classes of abelian varieties of dimension $g$ over $\Fq$. 
Table \ref{T:number-of-classes} shows the range of $g$ and $q$ covered by the database as well as the total number of isogeny classes of a given dimension. 

\begin{table}[h]
	\begin{center}
	\begin{tabular}{ccc}
		Dimension & Bound on $q$ & Isogeny class count \\  \hline
		1  & 499 &  6184 \\
		2 & 211 &  1253897\\
		3 & 25 & 1055307\\
		4 & 5 & 183607 \\
		5 & 3 & 281790\\
		6 & 2 & 164937 \\
	\end{tabular}
	\end{center}
\caption{\label{T:number-of-classes}The number of isogeny classes for each dimension.}
\end{table}
We note that by ``Bound on $q$" we mean that the set of isogeny classes of abelian varieties over $\Fq$ has been computed for every prime power less than or equal to $q$.
In dimensions $g=1$ and $2$, isogeny classes of abelian varieties defined over $\Fq$ for powers of $2, 3, 5,$ and $7$ up to $1024$ have also been included.
The isogeny class count is the total count of isogeny classes in the database, including for $q$ a power of $2,3,5,$ and $7$ up to $1024$ in dimensions $1$ and $2$.

For each isogeny class, we report a variety of invariants including the following:
  \begin{itemize}
   \item L-polynomial;
   \item Newton polygon (and hence $p$-rank and ordinarity);
   \item endomorphism algebra;
   \item Frobenius angles and angle rank;
   \item whether the isogeny class contains a principally polarized abelian variety or even a Jacobian (when known);
   \item whether the isogeny class is simple;
   \item when it is simple, the number field defined by the L-polynomial, and the Galois group of its splitting field;
   \item when it is not simple, the simple isogeny factors appearing in its decomposition;
   \item whether the isogeny class is geometrically simple;
   \item the point counts over small extensions of the base field;
   \item if the isogeny class contains a Jacobian, the point counts for its curve over small extensions of the base field (if the isogeny class is not known to contain a Jacobian, the same computation is performed and referred to as the point counts of the ``virtual curve" associated to this isogeny class);
   \item whether the isogeny class is a base change of an isogeny class defined over a smaller field (i.e., whether the isogeny class is primitive), and if it is not primitive, the primitive isogeny classes for which it is a base change;
   \item the twists of the isogeny class: the isogeny classes to which it becomes isogenous after a base change.
  \end{itemize}
Some of the questions we kept in mind in our analysis of the data are:
\begin{itemize}
\item
How does the number of objects in the database vary with $g$ and $q$?
\item
How are these objects distributed if we sort with respect to the Galois group, Newton polygon, or angle rank?
\item
How is the number of points on the abelian variety distributed? Does this change if we fix other invariants?
\item
What are the extreme values of the number of points?
\end{itemize}

There are, of course, many more questions one can ask, but some would require further data compilation; this is especially true for questions regarding the distinction between Jacobians of curves and more general abelian varieties.
We discuss possible future directions of inquiry at the very end of the paper (\S\ref{S:future directions}).

We conclude this introduction by directing the reader to some of the highights of the paper.

\begin{itemize}
	\item A tabulation of Galois groups of Weil polynomials (\S \ref{subsec:Galois groups}). We attempt
	to explain the results using Malle's heuristics on the distribution of Galois groups of number fields, 
	but the constants appearing in this method do not fit well with the data.
	\item Verification of some cases of a conjecture by Pries on the existence of abelian varieties over $\FF_p$ with prescribed $p$-ranks
	(Problem~\ref{problem: Pries p-ranks}).
	\item A list of possible combinations of $p$-rank, angle rank and Galois groups for a fixed $p$-rank (\S\ref{subsec:angle rank}). This is closely related to the Tate conjecture for abelian varieties,
	as in the work of Lenstra and Zarhin \cite{Lenstra1993}.
	\item Some analysis of the endomorphisms algebras for abelian varieties over $\FF_q$ for small $g$ and $q$, in line with a question of Oort (\S\ref{subsec:endomorphism algebras}).
	\item An analysis of point counts on abelian varieties on average over isogeny classes, rather than over isomorphism classes (\S\ref{subsec:sato-aint}). The data suggests a fit not to the Sato-Tate distribution, but an alternate distribution which we have not found in the literature (called herein the isogeny Sato-Tate distribution).
	\item A characterization of maximal and minimal abelian varieties, including an explanation of how they are unrelated to maximality and minimality of curves, as well as some open questions regarding simple abelian varieties (\S\ref{subsec:maximal-minimal}).
	\item An example of a supersingular curve of genus 5 in characteristic $3$, whose existence we were unable to infer from any general constructions (Example~\ref{exa:hyperelliptic supersingular genus 5}).
	\item Some counterexamples against a conjecture by Ahmadi-Sparlinski concerning the angle rank of ordinary Jacobians (Example~\ref{exa:Ahmadi-Shparlinski}).
\end{itemize}

\subsection*{Acknowledgements}
The authors began this project during the semester program ``Computational aspects of the Langlands program'' held at ICERM in fall 2015.
Subsequent workshops in support of the project were sponsored by American Institute of Mathematics
and the Simons Foundation. 
In addition Dupuy was partially supported by the European Research Council under the European Unions Seventh Framework
Programme (FP7/2007-2013) / ERC Grant agreement no. 291111/ MODAG while working on this project.
Kedlaya was supported by NSF (grants DMS-1501214, DMS-1802161), IAS (visiting professorship 2018--2019), UCSD (Warschawski Professorship), and a Guggenheim Fellowship (fall 2015).  
Roe was supported by Simons Foundation grant 550033.
Vincent was partially supported by NSF (grant DMS-1802323).
We would especially like to thank Everett Howe for implementing software for detecting principal polarizations in isogeny classes and for his feedback on an earlier version of this manuscript.  
Thanks also to Borys Kadets for computing bounds on point counts of abelian varieties over $\FF_5$ and $\FF_7$.

\section{Background}
\label{S:background}

In this section, we recall various standard facts about abelian varieties and their zeta functions.

\subsection{Weil Numbers, Characteristic Polynomials of Frobenius, and Zeta Functions}
Here we follow \cite{Fontaine2008}*{pg. 9} (up to a sign convention). 
Let $\QQ^{\alg}$ be an algebraic closure of $\QQ$, $\ZZ^{\alg}$ be the ring of algebraic integers in $\QQ^{\alg}$, and $w\in \ZZ$. 
A \emph{Weil $q$-number of weight $w$}
is an element $\alpha \in \QQ^{\alg}$ satisfying:
\begin{enumerate}
 \item there exists $i\in \NN$ such that $q^i\alpha \in \ZZ^{\alg}$; and
 \item for any embedding $\psi \colon \QQ^{\alg} \to \CC$, $\vert \psi (\alpha)\vert = q^{w/2}$.
\end{enumerate}
Hereafter we only consider Weil $q$-numbers which are \emph{effective}, meaning that $\alpha \in \ZZ^{\alg}$
(and hence $w \geq 0$).

Such numbers arise from the zeta functions of varieties over finite fields as follows:
Let $X$ be a variety over $\Fq$, a finite field of cardinality $q$ and characteristic $p$. Then we define 
$$Z(X/\Fq,T) := \exp\left( \sum_{n\geq 1} \# X(\FF_{q^n}) \frac{T^n}{n} \right).$$
For any Weil cohomology theory $X \mapsto H^w(X)$ (e.g., $\ell$-adic \'etale cohomology for $\ell \neq p$,
or Berthelot's $p$-adic rigid cohomology), we have
$$
Z(X/\Fq,T) = \prod_{w=0}^{2\dim(X)} L_w(T)^{(-1)^{1+w}}, \qquad L_w(T) := \det(1 - F T | H^w(X)),
$$
where $F \colon X \to X$ denotes the action of Frobenius.

 If $X$ is smooth and proper of dimension $n$, then the eigenvalues of $F$
on $H^w(X)$ are Weil $q$-numbers of weight $w$; for \'etale cohomology this is Deligne's theorem \cite{Deligne1974},
while for rigid cohomology it follows from Deligne's theorem by a result of Katz--Messing \cite{Katz1974}.
Also, if we define $\zeta(X,s) := Z(X/\Fq,q^{-s})$, then we have a functional equation
$$\zeta(X,n-s) = \pm (q^{n/2} q^{-s})^{\chi_{top}} \zeta(X,s) $$
which follows from Poincar\'e duality (see \cite{Freitag1988}*{\S II.1} for the \'etale case and
\cite{Kedlaya2006a} for the $p$-adic case). 
In the displayed formula above, $\chi_{top}$ is the topological Euler characteristic; if $X$ is the reduction of a variety over a number field then $\chi_{top}$ is the alternating sum of the Betti numbers of the cohomology of the complex manifold obtained by taking the $\CC$-points of the variety in characteristic zero. 

In the special case where $X=A$ is an abelian variety of dimension $g$, 
$\dim H^1(A) = 2g$ and $H^w(A) = \wedge^w H^1(A)$ for $w=0,\dots,2g$. We summarize this second relation by
saying that $L_w(T) = \wedge^w L_1(T)$. We refer to $L_1(T)$ as the \emph{$L$-polynomial} of $A$ and to its reverse $\det(T - F | H^1(X))$ as the \emph{characteristic polynomial} of $A$, or the \emph{Weil polynomial} of $A$. We use these two terms interchangeably throughout the text.

\subsection{Weil Polynomials}

Define a \emph{Weil $q$-polynomial} (or Weil polynomial when $q$ is understood)
to be a monic polynomial over $\ZZ$ whose roots are all Weil $q$-numbers of weight $1$. 
(Since we are exclusively interested in abelian varieties, we consider only Weil $q$-numbers of weight 1 in what follows.)
A standard first step in classifying Weil polynomials is the following observation left as an exercise to the reader. 

\begin{proposition}
\label{E:weil-q-number-example}
Let $f\in \QQ[T]$ be an irreducible polynomial with all real roots. 
Let $\beta \in \RR$ satisfy $f(\beta)=0$. 
Choose $q=p^n$ such that for all $\psi \colon \QQ(\beta) \to \RR$ we have  $\psi(\beta)^2 - 4q <0.$
If $\pi$ is a zero of $T^2 - \beta T + q=0$, then $\pi$ is a Weil $q$-number of weight $1$.  
\end{proposition}

As a converse, we have the following characterization of Weil $q$-numbers.

\begin{lemma}[Weil $q$-number characterization]
\label{L:weil-q-number-characterization}
Let $\pi$ be a Weil $q$-number of weight $1$.
\begin{enumerate}
\item\label{I:real-case} If $\QQ(\pi)$ has a real embedding, then $\pi=\pm\sqrt{q}$.
\item Suppose that $\psi \colon \QQ(\pi) \to \CC$ is a non-real embedding, then
\begin{enumerate}
\item $\beta:= \pi + q/\pi$ is totally real;
\item $\QQ(\pi)$ is a CM-field with maximal totally real subfield $\QQ(\beta)$, and $\pi$ has no real embeddings;
\item $\pi$ is a solution of $T^2-\beta T+ q$ where $\psi(\beta^2 - 4q)<0$ for all $\psi \colon \QQ(\pi) \to \CC$.
\end{enumerate}
\end{enumerate}
\end{lemma}
\begin{proof}
\begin{enumerate}
\item In this case, $\psi(\pi)^2 = \psi(\pi)\overline{\psi(\pi)} = q$. 
\item 
\begin{enumerate}
\item Since $\pi$ has absolute value $\sqrt{q}$ in every embedding, $q/\pi$ is its complex conjugate.
\item The fact that $\QQ(\pi)$ is CM is \cite{Honda1967}*{Proposition 4}, and the second part follows.
\item Consider the equation 
$$T^2 - \beta T + q=0.$$
Let $\gamma = \pi-q/\pi$ and note that $\beta^2 - 4q = \gamma^2$.
The quadratic formula gives $\frac{1}{2}( \beta \pm \gamma)$ as the solutions, which simplify to $\pi$ and $q/\pi$.
The second half of the statement follows from the fact that $\pi$ has no real embeddings.
\end{enumerate}
\end{enumerate}
\end{proof}

\subsection{Weil Polynomials and Isogeny Classes of Abelian Varieties}
\label{subsec:Honda-Tate}

For most classes of varieties over a fixed finite field, among the possible zeta functions consistent with the Weil conjectures, it is difficult to predict in advance exactly which ones occur. However, for abelian varieties, this question is completely solved by the Honda-Tate theorem; we follow the treatment of this result given in \cite{Waterhouse1971}. 

Let $A$ and $B$ be abelian varieties over $\Fq$ with characteristic polynomials
$$
P_A(T) := \det(T - F\vert H^1(A)), \qquad P_B(T) := \det(T-F\vert H^1(B));
$$
these are Weil $q$-polynomials.
The Honda-Tate theorem makes the following assertions:
\begin{itemize}
\item
$P_A$ divides $P_B$ if and only if $B$ is isogenous (over $\FF_q$) to a product in which $A$ occurs as a factor; in particular, $A$ and $B$ are isogenous if and only if $P_A = P_B$. 
\item
If $A$ is simple, then $P_A = h_A^{e}$ for some irreducible $h_A(T) \in \ZZ[T]$
and some positive integer $e = e_A$ which can be read off explicitly from $h_A$ (see below).
\item
Every irreducible Weil $q$-polynomial occurs as $h_A$ for some $A$ (which is unique up to isogeny).
\end{itemize}

We explain the rule for computing $e_A$ for $A$ simple in the context of analyzing the $\QQ$-endomorphism algebra $E := \End^0(A/\Fq)$.
Let $\pi$ be the class of $T$ in $\QQ(\pi) := \QQ[T]/(h_A(T))$, and identify $\pi$ with the action of Frobenius in $E$. Then $E$ is a division algebra with center $\QQ(\pi)$
and $$2g = [\QQ(\pi):\QQ] e_A, \qquad e_A = \sqrt{[E:\QQ(\pi)]}.$$
The Brauer invariant $\inv_v(E)$ of $E$ at a place $v$ of $\QQ(\pi)$ is given as follows:
\begin{itemize}
\item
For $v$ finite and not lying above $p$, $\inv_v(E)  = 0$.
\item
For $v$ finite and lying above $p$,
 \begin{equation*}
\inv_v(E) \equiv - (\log_q \left| \pi \right|_v) [K_v:\QQ_p] \pmod{\ZZ}.
 \end{equation*}
\item
For $v$ archimedean, $\inv_v(E) = \frac{1}{2}$ if $v$ is real (which by Lemma~\ref{L:weil-q-number-characterization}
means
$\pi = \pm \sqrt{q}$) and 0 otherwise.
\end{itemize}
The exponent $e_A$ can now be computed as
\[
e_A = \lcd_v\{\inv_v(E)\},
\]
where $\lcd$ denotes the least common denominator. In particular, $e_A = 1$ whenever $A$ is ordinary or ``almost ordinary'' (see \S\ref{subsec:Newton polygons}).

\subsection{Newton Polygons, \texorpdfstring{$p$}{p}-rank and Ordinarity}
\label{subsec:Newton polygons}
Let $A$ be an abelian variety over $\Fq$ and define $P_A$ as in \S\ref{subsec:Honda-Tate}.
The \emph{normalized Newton polygon} of $P_A$ is the lower convex hull of the set
\[
\left\{ \left(i, \frac{\ord_p(a_i)}{\ord_p(q)} \right): i=0,\dots,2g \right\}
\]
where $P_A(T) = T^{2g} + a_1 T^{2g-1} + \cdots + a_{2g}$.
(This is invariant under base extension.)
The normalized Newton polygon of $P_A$ is the graph of a piecewise linear function on $[0,2g]$ with changes of slope only at integer values. In particular, on each of the intervals $[0,1],\dots,[2g-1,g]$ the function has a unique slope; we call these the \emph{slopes} of $P_A(T)$. For any place $v$ of $K = \QQ(\pi)$ above $p$,
the slopes of $P_A(T)$ coincide with the ratios $v(\alpha)/v(q)$ as $\alpha$ varies over the roots of $P_A$.

The normalized Newton polygon of $A$ satisfies the following properties:
\begin{itemize}
\item
The left endpoint is $(0,0)$ and the right endpoint is $(2g,g)$.
\item
The vertices are all lattice points with nonnegative second coordinate.
\item
The vertices are symmetric: $(i,j)$ is a vertex if and only if $(2g-i,g-j)$ is a vertex. Equivalently,
$(i,j)$ lies above the polygon if and only if $(2g-i,g-j)$ does so.
\end{itemize}
We say that any convex polygon satisfying these conditions is \emph{eligible in dimension $g$}. For $N,N'$ two eligible polygons, we write $N \leq N'$ if $N'$ lies on or above $N$, and $N < N'$ if $N \leq N'$ and $N \neq N'$.
The eligible Newton polygons in dimension $g$ form a partially ordered set with unique minimal and maximal elements.
The minimal polygon is the one with vertices  $(0,0), (g,0), (2g, g)$, in which the slopes are $0,1$ (each with multiplicity $g$); when this occurs we say that $A$ is \emph{ordinary}. 
If we exclude the ordinary Newton polygon, then there is again a unique minimal polygon, with vertices $(0,0), (g-1,0), (g+1,1), (2g,g)$; when this occurs we say that $A$ is 
\emph{almost ordinary}, as in \cite{Lenstra1993}.
The maximal polygon is the one with vertices $(0,0), (2g,g)$, in which the slopes are $\frac{1}{2}$ (with multiplicity $2g$); when this occurs we say that $A$ is \emph{supersingular}.

Define the \emph{elevation} of an eligible Newton polygon $N$ in dimension $g$ 
as the number of lattice points $(i,j)$ with 
$1 \leq i \leq g, j \geq 0$ which lie strictly below $N$.
(This is not standard terminology.)

\begin{lemma} \label{L:add to chain}
Let $N,N'$ be eligible Newton polygons in dimension $g$ with $N < N'$. Then there exists an eligible Newton polygon $N''$ in dimension $g$ such that $N < N'' \leq N'$ and the elevation of $N''$ is one more than that of $N$.
\end{lemma}
\begin{proof}
Since $N < N'$, there must exist a vertex $(i,j)$ of $N$ lying strictly below $N'$; by symmetry,
$(2g-i,g-j)$ also lies strictly below $N'$. Let $N''$ be the lower convex hull of the set of lattice points
lying on or above $N$ exclusive of $(i,j)$ and $(2g-i,g-j)$; this choice has the desired form.
\end{proof}
\begin{corollary} \label{cor:NPcat}
The poset of eligible Newton polygons in dimension $g$ is \emph{catenary}: any two maximal chains between the same endpoint have the same length (namely the difference in elevation).
\end{corollary}

For any positive integer $d$, by \cite{Katz1979} the Newton polygon defines a locally closed stratification on the coarse moduli space $\A_{g,d}$ of $g$-dimensional abelian varieties equipped with a polarization of degree $d^2$.

\begin{theorem} \label{T:NP stratification}
Let $N$ be an eligible Newton polygon of dimension $g$.
\begin{enumerate}
\item
There is a (nonempty) stratum of $\A_{g,d}$ with Newton polygon $N$.
\item
Each irreducible component of this stratum has codimension equal to the elevation of $N$.
\item
If $N$ is not the supersingular Newton polygon, then the corresponding stratum is geometrically irreducible.
\end{enumerate}
\end{theorem}
\begin{proof}
For the supersingular Newton polygon, this is contained in \cite[Theorem~4.9]{Li1998}.
For general $N$, apply Lemma~\ref{L:add to chain} repeatedly to construct a maximal chain $N_0 < \cdots < N_1$ containing $N$, where $N_0$ and $N_1$ are the ordinary and supersingular Newton polygons. By the supersingular case, the length of the chain (which is the elevation of $N_1$) equals the codimension of each irreducible component of the supersingular stratum.

We now appeal to the de Jong--Oort purity theorem
\cite[Theorem~4.1]{DeJong2000}, which asserts that the Newton polygon stratification on $\A_{g,d}$ jumps purely in codimension 1.
For each Newton polygon $N$ in the chosen chain, the union $X_N$ of all strata corresponding to Newton polygons on or above $N$ is a closed subscheme of $\A_{g,d}$; the codimension of $X_N$ starts at 0 for $N = N_0$,
increases by at most 1 at each step, and ends with the elevation of $N_1$. Consequently, it must increase by exactly 1 at each step; from this, the first two claims follow.
The third claim is a theorem of Chai--Oort \cite[Theorem~A]{Chai2011}.
\end{proof}

\subsection{Galois Groups} 
Let $\pi$ be a Weil $q$-number.
By Lemma \ref{L:weil-q-number-characterization}, either $\pi = \pm \sqrt{q}$ or $\QQ(\pi)$ is CM;
assume hereafter that the second case occurs.
Write $\QQ(\pi)^{\gal}$ for the Galois closure of $\QQ(\pi)$ and set $2d = [\QQ(\pi) : \QQ] = \frac{2g}{e}$.
Let $G = \Gal(\QQ(\pi)^{\gal}/\QQ)$, considered as a subgroup of $S_{2d}$ by its action on the conjugates of $\pi$.

\begin{lemma} \label{lem:CMGal} Let $G$ be the Galois group of a Weil $q$-number $\pi$.
\begin{enumerate}
\item $G$ is a subgroup of $W_{2d} := C_2^d \rtimes S_d$ which acts transitively on the $2d$ roots, where here $W_{2d}$ is the wreath product of $C_2$ by $S_d$. 
\item $G$ contains complex conjugation, which is the unique nontrivial element of the center of $W_{2d}$.
\end{enumerate}
\end{lemma}

\begin{proof}
The fact that $G \subseteq W_{2d}$ and contains complex conjugation follows from \cite{Dodson1984}*{Prop 1.1}.  Transitivity follows from the fact that $\QQ(\pi)$ is a field, the claim that the center of $W_{2d}$ has order 2 follows from its presentation as a semidirect product, and that complex conjugation is central is \cite{Honda1967}*{Prop 1 (b)}.
\end{proof}

Note that we usually have $d=g$.  In this case, we conjecture that Lemma \ref{lem:CMGal} gives the only constraints on $G$:

\begin{conjecture}
Suppose that $G$ is a transitive subgroup of $W_{2g}$ containing complex conjugation.  Then there is a Weil $q$-number $\pi$ such that $\Gal(\QQ(\pi)^{\gal}/\QQ) \cong G$.
\end{conjecture}

Note that the condition on containing complex conjugation is necessary: when $g=4$ there is a transitive subgroup of $W_{8}$ (with transitive label \texttt{8T14} and abstractly isomorphic to $S_4$) that does not contain complex conjugation and thus does not arise as the Galois group of a CM-field.

\subsection{Frobenius Angle Rank} \label{ssec:angle rank2}
For $A$ an abelian variety of dimension $g$ with $L$-polynomial $L(T) = \prod_{i=1}^{2g}(1-\alpha_i T)$, the \emph{angle rank} of $A$ is the quantity
 $$ \delta(A) = \dim_{\QQ}(\Span_\QQ(\lbrace \arg(\alpha_i): 1\leq i \leq 2g  \rbrace \cup \lbrace \pi \rbrace )) - 1 
 \in \{0,\dots,g\}.$$
Equivalently, $\delta(A)$ equals the number of multiplicatively independent elements in the set $\{\alpha_1,\dots,\alpha_{2g}, q^{1/2}\}$, which equals the rank of the
Frobenius torus associated to $A$ \cite{Chi1992}.
 
The angle rank detects multiplicative relations among the roots of $L$; these are closely related to exceptional Hodge classes on powers of $A$.\footnote{A \emph{Hodge class} is an \'{e}tale cohomology class that is invariant under the twisted action of Frobenius. A Hodge class is \emph{exceptional} if it is not in the ring (under cup product) generated by Hodge classes of weight two.
The Tate conjecture states that the space of Hodge classes is spanned by the images of algebraic cycles under the cycle class map to $l$-adic cohomology. 
This remains an open problem.} For example, by a theorem of Zarhin \cite[Theorem~3.4.3]{Zarhin1994}, $\delta(A) = g$ if and only if there are no exceptional Hodge classes on any power of $A$. On the other end, $\delta(A) = 0$ if and only if $A$ is supersingular; this follows from the fact, proven here in Example \ref{exa:supersingular}, that an abelian variety $A$ is supersingular if and only if $\pi_A = \zeta \sqrt{q}$ for $\zeta$ a root of unity.

Another useful observation is that if $A$ is simple and $0 < \delta(A) < g$, then the inclusion
$G \subseteq W_{2d}$ from Lemma~\ref{lem:CMGal} must be strict, because the action of $G$ preserves $\QQ$-linear relations between $\pi$ and the $\arg(\alpha_i)$. For a more refined version of this statement, and more about angle rank, see \cite{Dupuy2020b}.

\subsection{Bounds on Point Counts} \label{ssec:point_count_bounds}

Recall that if $A$ is an abelian variety of dimension $g$ over $\FF_q$ with characteristic polynomial
$(T- \alpha_1) \cdots (T- \alpha_{2g})$, then
\[
\#A(\FF_q) = (1-\alpha_1) \cdots (1 - \alpha_{2g}).
\]
The Weil bound implies
\[
(\sqrt{q}-1)^{2} \leq \#A(\FF_q)^{1/g} \leq (\sqrt{q}+1)^{2};
\]
this can be sharpened a bit to give 
\[
\lceil (\sqrt{q}-1)^2 \rceil \leq \#A(\FF_q)^{1/g} \leq \lfloor (\sqrt{q}+1)^2 \rfloor
\]
(see \cite[Th\'eor\`eme~1.1]{Aubry2012} or \cite[Corollary~2.2, Corollary~2.14]{Aubry2013}).
For general $A$, these bounds are best possible: for $g=1$, it follows from the Honda-Tate theorem (or an earlier theorem of Deuring) that $\#A(\FF_q)$ can take the values $q + 1 \pm \lfloor 2\sqrt{q} \rfloor$,
and the same is then true for arbitrary $g$ by taking powers.

In light of the fact (which we just used) that
\[
\#(A_1 \times_{\FF_q} A_2)(\FF_q) = \#A_1(\FF_q) \times \#A_2(\FF_q),
\]
it is natural to separately consider what happens when $A$ is required to be simple. In this case, if we exclude $A$ of low dimension, the Weil bounds can be sharpened further.
\begin{theorem}[{\cite[Theorem~2.5]{Kadets2019}}]   \label{T:Kadets general bound}
If $A$ is simple of dimension $g>1$, then
\[
\lfloor (\sqrt{q}-1)^2 \rfloor + 1 \leq \#A(\FF_q)^{1/g} \leq \lceil (\sqrt{q}+1)^2 \rceil - 1.
\]
\end{theorem}
Note that this is an improvement on the bound stated earlier when $q$ is a square.

For small $q$, one can make some additional improvements:
\begin{theorem}[{\cite[ Theorem~3.2]{Kadets2019}}] \label{T:Kadets small-q bound}
For values of $q$ listed below, one has the following lower and upper bounds on $\#A(\FF_q)^{1/g}$ for $g \geq 4$.
(More precisely, the bounds hold for all $g$ outside of an explicit finite set of exceptional choices of $A$, each of which has dimension at most $3$; see \cite[Table~2, Table~3]{Kadets2019}.)
\[
\begin{array}{c|ccccc}
q & 2 & 3 & 4 & 8 & 9 \\
\hline
\mbox{lower bound} & 1 & 1.359 & 2.2750 & 4.635 & 5.47 \\
\mbox{upper bound} & 4.0347 & 5.6333 & 7.3818 & 13.05 & 14.303 
\end{array}
\]
\end{theorem}

Looking at an asymptotic version of the question, one can see that there is not much room left for potential improvement in these bounds.
Building on \cite{Aubry2013}, Kadets has shown the following:
\begin{theorem}[{\cite[Proposition~1.4, Theorem~1.6]{Kadets2019}}] \label{T:Kadets asymptotic bound}
Define
\[
a(q) := \liminf_A \#A(\FF_q)^{1/g}, \qquad A(q) := \limsup_A \#A(\FF_q)^{1/g},
\]
where in both cases $A$ varies over simple abelian varieties of dimension $g$ over $\FF_q$.
Then
\[
\lfloor (\sqrt{q} - 1)^2 \rfloor + 1 \leq a(q) \leq \lceil (\sqrt{q}-1)^2 \rceil + 2, \qquad
\lfloor (\sqrt{q} + 1)^2 \rfloor - 2 - q^{-1} \leq A(q) \leq \lceil (\sqrt{q}+1)^2 \rceil -1.
\]
\end{theorem}

The cases where $\#A(\FF_q) = 1$ have been classified by 
Madan--Pal  (modulo Remark~\ref{R:Madan-Pal}).
\begin{theorem}[{\cite[Theorem~4]{Madan1977}}] \label{T:Madan-Pal}
Suppose that $A$ is simple of dimension $g$ and that $\#A(\FF_q) = 1$.
\begin{itemize}
\item
We must have $q \leq 4$. (This is easy: if $q \geq 5$, then $\sqrt{q}-1 > 1$ and the Weil bound implies the claim.)
\item
If $q = 3$ or $q=4$, then $g=1$. (For $q=4$, this is the equality case of the Weil bound. For $q=3$ a more careful argument is needed.)
\item
If $q=2$, then the characteristic polynomial of $A$ belongs to an explicit (infinite) list enumerated in \emph{op.\ cit.\ }
\end{itemize}
\end{theorem}

\begin{remark} \label{R:Madan-Pal}
In \cite[Theorem~4]{Madan1977}, the case $q=2$ of Theorem~\ref{T:Madan-Pal} is asserted modulo a conjecture of
Robinson \cite{Robinson1964}: for $n$ a positive integer and $\zeta_n = e^{2\pi i /n}$, $\zeta_n^2 + 6 \zeta_n + 1$ is not a square in $\QQ(\zeta_n)$ unless $n=7$ or $n=30$. 
This was confirmed by Robinson \cite{Robinson1977} using a method of  Cassels \cite{Cassels1969} and Loxton \cite{Loxton1972}.
\end{remark}

\begin{remark}
A closely related question to Theorem~\ref{T:Madan-Pal}
is to classify simple abelian varieties $A$ 
of genus $g$ over $\FF_q$ for which $\#A(\FF_q) = \#A(\FF_{q^n})$ for some $n>1$
(meaning that $A$ acquires no new points over $\FF_{q^n}$).
In \cite[Corollary~3.4]{Kadets2019}, it is shown 
(using Theorem~\ref{T:Kadets general bound} and Theorem~\ref{T:Kadets small-q bound})
that this cannot occur for $n\geq 4$;
for $n=3$ it can only occur for $q = 2, g=1$;
and for $n=2$ it occurs precisely when $A$ is the quadratic isogeny twist (see Definition~\ref{D:isogeny twist}) of one of the cases listed in Theorem~\ref{T:Madan-Pal}.

This in turn implies a corresponding classification for curves over $\FF_q$ which acquire no new points over $\FF_{q^n}$, using the classification of function fields of class number one
(see Example~\ref{exa:class number one}).
\end{remark}

\section{Algorithms}
\label{S:algorithms}

We summarize some of the main algorithms used to compute data about abelian varieties in the LMFDB.
The code \cite{abvarfq} used is available at

\begin{center}
\url{https://github.com/LMFDB/abvar-fq}.
\end{center}

\subsection{Enumerating Weil Polynomials} \label{ssec:enumerating Weil polynomials}

We begin with the algorithm for enumerating Weil $q$-polynomials of fixed degree for a fixed prime power $q$.
This algorithm originated in \cite{Abbott2006} and was described in more detail in
\cite[Section~5]{Kedlaya2008} and \cite[Section~2]{Kedlaya2015}, to which the reader is referred for more details.

\begin{remark}\label{rem:iterator}
The code for enumerating Weil polynomials was incorporated into \texttt{Sage-9.1.beta0} \cite{sage} in January 2020.
\end{remark}

As noted earlier, identifying Weil $q$-polynomials of degree $2g$
\[
P(T) = T^{2g} + a_1 T^{2g-1} + \cdots + a_{2g}
\]
is equivalent to identifying polynomials of degree $g$ with integer coefficients
\[
Q(T) = T^g + b_1 T^{g-1} + \cdots + b_g
\]
whose roots are all real and belong to the interval $[-2\sqrt{q}, 2\sqrt{q}]$, via the relation
\[
P(T) = T^g Q(T + q/T).
\]
In particular, for any integer $i \in \{1,\dots,g\}$, $a_1,\dots,a_i$ uniquely determine $b_1,\dots,b_i$ and vice versa. The strategy is to catalog these polynomials recursively: given a putative choice of $a_1,\dots,a_{i-1}$
(or equivalently $b_1,\dots,b_{i-1}$),
use various techniques to impose necessary conditions on $a_{i}$ (or equivalently $b_{i}$), 
then recurse on the remaining options.
Some of these necessary conditions are as follows:
\begin{itemize}
\item
Rolle's theorem: the roots of $Q^{(g-i)}(T)$ must all be real and belong to $[-2\sqrt{q}, 2\sqrt{q}]$. This can be verified by computing a subresultant (Sturm-Habicht) sequence.
\item
Connectivity: the set of values of $b_i$ satisfying the previous condition is either a closed interval or empty.
\item
Power sums: the $i$-th power sum of the roots of $P(T)$ must have absolute value at most $q^{i/2}$. 
\item
Descartes's rule of signs: the polynomials $Q(T + 2\sqrt{q})$ and $(-1)^g Q(-2\sqrt{q} - T)$ have all roots real and nonpositive, so their coefficients must be nonnegative.
\item
Hamburger criterion: for $s_0, s_1, \dots$ the sequence of power sums of $Q$, the Hankel matrix
\[
\begin{pmatrix}
s_0 & s_1 & s_2 & \cdots \\
s_1 & s_2 & s_3 & \cdots \\
s_2 & s_3 & s_4 & \cdots \\
\vdots & \vdots & \vdots & \ddots
\end{pmatrix}
\]
is nonnegative definite (equivalently, its principal minors are nonnegative).
\item
Hausdorff criterion: for all $i,j$, the sum of $(2\sqrt{q}-x)^i (2\sqrt{q}+x)^j$ as $x$ varies over the roots of $Q$ is nonnegative.
\end{itemize}

A crucial but counterintuitive point is that while some subsets of these conditions together form sufficient constraints to identify Weil $q$-polynomials, no subset gives sufficient ``on-line'' criteria. That is, they cannot be applied directly to detect whether a given sequence $a_1,\dots,a_k$ can be extended to give the coefficients of some Weil polynomial; the conditions can only be checked once a full sequence has been generated. Therefore we impose ``more" conditions than are sufficient to cut down the search space as much as possible as early as possible in the computation.

One serious issue with this calculation is that, while it is trivial to check the output for false positives (e.g., Sage has a built-in function \texttt{is\_weil\_polynomial} to test whether a given integer polynomial is a Weil polynomial), it is rather difficult to check for false negatives other than by generating Weil polynomials using another method (e.g., by computing zeta functions of abelian varieties) and checking for their presence. One useful consistency check is to run the computation for $q=1$, where the answer is known: by Kronecker's theorem, the irreducible Weil $q$-polynomials for $q=1$ are precisely the cyclotomic polynomials.  One can also attempt to use the explicit descriptions of the spaces of Weil polynomials given for $g=3$ in \cite{Haloui2010}, for $g=4$ in \cite{Haloui2012}, and for $g=5$ in \cite{Sohn2013}.  Our results mostly matched in the $g=3$ case (we found a few extra non-simple examples), while the conditions in \cite{Haloui2012}*{Thm. 1.1} and \cite{Sohn2013}*{Thm. 2.1} were not actually satisfied by all of the Weil polynomials that we found.

\subsection{Point Counts}

As noted earlier, for $A$ an abelian variety of dimension $g$, we have $H^i(A) = \wedge^i H^1(A)$ for $i=1,\dots,2g$, so the zeta function of $A$ is determined completely by the $L$-polynomial $L(T)$. In particular, 
if we write $L(T) = \prod_{i=1}^{2g} (1 - \alpha_i T)$, then
\[
\#A(\Fq) = \prod_{i=1}^{2g} (1 - \alpha_i) = L(1).
\]
Similarly, for any positive integer $r$,
\[
\#A(\FF_{q^r}) = \prod_{i=1}^{2g} (1 - \alpha_i^r) = \Res(L(T), T^r - 1).
\]
In particular,
\[
\#A(\FF_{q^2}) = L(1) L(-1).
\]

\subsection{Curve Point Counts}
\label{subsec:curve counts}

If $A$ is isogenous to the Jacobian of a curve $C$, then the sequence $c_n := \#C(\FF_{q^n})$ satisfies
\[
\frac{L(T)}{(1-T)(1-qT)} = \exp\left( \sum_{n=1}^\infty \frac{c_n}{n} T^n \right).
\]
In fact, whether or not $A$ is isogenous to a Jacobian, the sequence $c_n$ consists of integers, which 
we report as the point counts of a ``virtual curve'' with Jacobian $A$.
If the $c_n$ violate any known constraints on the point counts of a curve of genus $g$, then
$A$ cannot be isogenous to a Jacobian; these include trivial constraints (such as $c_1 \geq 0$
and $c_{mn} \geq c_n$) and some less trivial ones (e.g., the Ihara bound\footnote{	Ihara's bound is the following: 
	$\#C(\Fq)/q \leq \frac{1}{2}( \sqrt{ 8q+1} -1 )$. \cite{Ihara1981} }).

\subsection{Base Change, Primitivity and Isogeny Twists} \label{ssec:base change}

If $A$ is an abelian variety over $\Fq$, then the Weil polynomials for $A$ and for the base change of $A$ to $\FF_{q^r}$ are related as follows.
\begin{proposition}
Suppose $P(T) = \prod_{i=1}^{2g} (T - \alpha_i)$ is the Weil polynomial associated to $A/\Fq$ by the Honda-Tate theorem.  Then the Weil polynomial associated to the base change of $A$ to $\FF_{q^r}$ is
\[
P_r(T) := \prod_{i=1}^{2g} (T - \alpha_i^r).
\]
\end{proposition}
\begin{proof}
The roots of $P(T)$ are the eigenvalues of the action of the $q$-Frobenius on $H^1(A)$, and the $q^r$-Frobenius is just the $r$th power of the $q$-Frobenius.
\end{proof}

\begin{definition} \label{D:isogeny twist}
We say that $A$ is \emph{primitive} if $A$ is not isogenous to the base change of any abelian variety defined over a subfield of $\Fq$.  We say that two abelian varieties $A$ and $B$ are \emph{isogeny twists} if they become isogenous after some finite extension of $\Fq$.  The \emph{isogeny twist class} of $A$ is the set of isogeny twists of $A$.
\end{definition}

\begin{corollary} \label{cor:simple_twist}
If $A$ and $B$ are simple abelian varieties of dimension $g$ over $\Fq$ with associated Weil numbers $\alpha$ and $\beta$, then $A$ and $B$ are isogeny twists if and only if $\alpha = \zeta \beta$ for some root of unity $\zeta$.
\end{corollary}

We use various methods for computing $P_r(T)$.
The simplest is to just factor $P(T)$ approximately over $\CC$, raise each root to the $r$th power and then recognize the coefficients of the product as integers.
In order to make this approach rigorous, one needs to use ball arithmetic in $\CC$ and increase the precision if there are multiple integers within the error bounds for any coefficient.

The second approach is to symbolically express the coefficients of $P_r(T)$ as polynomials in the coefficients of $P(T)$ using Newton's identities to change basis between elementary symmetric polynomials and power sums.
For small values of $r$ the resulting transformation is not difficult to compute but for values of $r$ larger than about $10$ the memory footprint of the algorithm grows rapidly.
We therefore only use this approach for smooth values of $r$, where cached transformations can be repeatedly applied to compute the overall base change.

Finally, the base change can be computed using polynomials over the cyclotomic field $\QQ(\zeta_r)$.
In particular, we can use the identity
\[
P_r(T^r) = \prod_{i=0}^{r-1} P(\zeta_r^i T)
\]
to determine $P_r(T)$. An equivalent approach is to compute $P_r(T)$ as the resultant of $P(U)$ and
$U^r-T$. 

Since we are enumerating Weil polynomials for many $q$, we can use these methods for computing
base changes to easily determine which isogeny classes are primitive, and to find primitive models
for those which are not.  We simply compute all base changes from $\Fq$ to $\FF_{q^r}$ when both
prime powers are contained in the database.

Determining which abelian varieties are isogeny twists of each other is a little more difficult.
The condition in Corollary \ref{cor:simple_twist} is difficult to determine directly from Weil polynomials.
We can improve it slightly as follows.

\begin{proposition}
Two simple abelian varieties $A$ and $B$ are isogenous over $\overline{\FF}_q$ if and only if there exists a number field $K$ containing Galois conjugates $\pi_A'$ and $\pi_B'$ of $\pi_A$ and $\pi_B$ such that  $\pi_A' \OO_K = \pi_B '\OO_K$.  
\end{proposition}
\begin{proof}
The requirement that $\pi_A'$ and $\pi_B'$ generate the same ideal is equivalent to $\pi_A' = u\pi_B'$ for some unit $u \in \OO_K$.  Then $u$ has absolute value $1$ at every finite place, and because $\lvert \pi_A'\rvert_v = \lvert \pi_B' \rvert_v = \sqrt{q}$ for every infinite place $v$, $u$ has absolute value $1$ at every infinite place as well.  Therefore $u$ is a root of unity by Kronecker and we may invoke Corollary \ref{cor:simple_twist}.
\end{proof}

This result is not sufficient for our purposes for two reasons.  First, it seems to require computations in Galois closures since we need to work with arbitrary conjugates of $\pi_A'$ and $\pi_B'$.  Second, as with Corollary \ref{cor:simple_twist}, it only applies when both $A$ and $B$ are simple,
yet it is possible for a simple abelian variety to be a isogeny twist of a non-simple one.

Instead, we break up the set of isogeny classes for each $g$ and $q$ into clusters based on invariants,
then use a pairwise test to further refine these clusters into isogeny twist classes.
Isogeny twists will have the same slopes and the same geometric endomorphism algebra,
whose computation we describe in the next section.
These two invariants are sufficient to divide up isogeny classes into clusters whose
size is already usually in the single digits, and is at most several hundred for
the values of $g$ and $q$ that we consider.
We then use the following result for each pair of isogeny classes in the cluster.
\begin{proposition}[{\cite{CMSV2019}*{Sec. 7.2}}] \label{prop:isodef}
Suppose $A$ and $B$ are abelian varieties over $\Fq$ with Weil polynomials $P(T)$ and $Q(T)$.
If there is an isogeny from $A$ to $B$ defined over $\FF_{q^r}$ then the cyclotomic polynomial $\Phi_r(T)$
divides the resultant $\Res_z(P(z), z^{2g} Q(T/z))$.
\end{proposition}
If we set $m$ to be the least common multiple of the orders of all cyclotomic polynomials dividing this resultant,
then we can determine whether $A$ and $B$ are isogeny twists by computing the base changes
$P_m(T)$ and $Q_m(T)$ and checking for equality. Note that to find the product of all cyclotomic polynomials dividing a given polynomial, it is not necessary to factor it fully; there is a more efficient
algorithm of Beukers--Smyth \cite[\S 2]{BS02} that finds cyclotomic factors using the Euclidean algorithm.
(In Sage, a polynomial over $\QQ$ has a method \texttt{cyclotomic\_part} implementing this algorithm.)

\subsection{Endomorphism Algebras}

The method for determining the Brauer invariants of the endomorphism algebra $\End^0(A/\Fq)$ is described in Section \ref{subsec:Honda-Tate}; given an irreducible Weil polynomial $h_A$ these invariants give the power $e_A$ such that $h_A^{e_A}$ is the characteristic polynomial of an isogeny class of abelian varieties over a finite field. We can also use Proposition \ref{prop:isodef} with $B=A$ to find all possible extension degrees
where the endomorphism algebra might change, and then compute the endomorphism algebra anew for each such base change. These two results allow us to, for an isogeny class defined over a finite field $\Fq$, give the endomorphism algebra of the base change of this isogeny class over any finite extension of $\Fq$.

We refer to the minimal extension over which all endomorphisms are defined as the \emph{endomorphism field};
one can search the database by degree of the endomorphism field, which we call the \emph{endomorphism degree}. 

For the convenience of the reader, we now explain how to describe the endomorphism algebra of a simple isogeny class $A$ over a finite field $\Fq$ given its Brauer invariants and its characteristic polynomial. This is enough to describe the endomorphism algebra of any isogeny class, as, if $A$ is isogenous to an abelian variety 
\begin{equation*}
A_1^{n_1} \times A_2^{n_2} \times \cdots \times A_r^{n_r}
\end{equation*}
where each $A_i$ is simple and $A_i$ is not isogenous to $A_j$ if $i \neq j$, then
\begin{equation*}
\End^0(A/\Fq) \cong M_{n_1}(\End(A_1/\Fq)) \times \cdots \times M_{n_r}(\End(A_r/\Fq)).
\end{equation*}
Assume now that $A$ is simple. We begin by determining the center of the endomorphism algebra of $A$: If the characteristic polynomial is $h_A^{e_A}$ for $h_A$ irreducible, then the center of $\End^0(A/\Fq)$ is the field $\QQ[x]/h_A$ generated by $\pi_A$. We then compute the degree of $\End^0(A/\Fq)$ over its center. Its square root is given by the order of the class of $\End^0(A/\Fq)$ in the Brauer group of its center, by \cite[p.~142]{Tate1966}. By the exact sequence
\begin{equation*}
0 \to \Br(F_A) \to \bigoplus_{v} \Br(F_{A,v}) \to \mathbb{Q}/\mathbb{Z} \to 0,
\end{equation*}
where $F_A=\QQ[x]/h_A$, this is the least common denominator of the Brauer invariants of  $\End^0(A/\Fq)$, or more simply put, $e_A$.

To complete our description of $\End^0(A/\Fq)$, we note that by \cite[Theorem~2 and its proof]{Tate1966}, if $\QQ[x]/h_A$ has a real place, then either $\QQ[x]/h_A = \QQ$ and $\End^0(A/\Fq)$ is the quaternion algebra over $\QQ$ that is ramified at $p$ and $\infty$, or $\QQ[x]/h_A = \QQ(\sqrt{p})$ and $\End^0(A/\Fq)$ is the quaternion algebra over $\QQ(\sqrt{p})$ ramified at both real places and nowhere else. 

Otherwise, $\QQ[x]/h_A$ is totally complex. If $e_A =1$, then $\End^0(A/\Fq)= \QQ[x]/h_A$. Otherwise, $\End^0(A/\Fq)$ is ramified only at places dividing $p$, and we identify the isomorphism class of $\End^0(A/\Fq)$  by its degree over $\QQ[x]/h_A$ and its Brauer invariants at the places above $p$ in $\QQ[x]/h_A$.

\subsection{Principal Polarizations}\label{S:principal-polarizations}

\emph{A priori}, the isogeny class associated to a Weil polynomial consists of unpolarized abelian varieties. 
In particular, to show how polarizations may vary within an isogeny class, we remind the reader of the following standard facts:
\begin{enumerate}
\item
For $g \geq 2$ and over an algebraically closed field, every abelian variety is isogenous to one that is principally polarizable. 
\item
For every field, there exist abelian varieties defined over that field which are not principally polarizable.
For example, over an algebraically closed field of characteristic zero, there is a simple proof that every principally polarizable abelian variety is isogenous to an abelian variety without a principal polarization.\footnote{
See \url{https://mathoverflow.net/questions/16992/non-principally-polarized-complex-abelian-varieties}.}
However, this last fact is not true in general, see Example \ref{E:isogenoustopp}.
\item
Polarization also behaves poorly under decomposition into isogeny factors: An abelian variety can be principally polarizable without all of its isogeny factors being principally polarizable.
See \cite[Example~13.8]{Howe1995} or \avlink{3.2.ac\_c\_ad}, which contains a Jacobian whose two-dimensional factor is not principally polarizable.
\end{enumerate}

The takeaway from these facts is that the quality of admitting a principal polarization is not an isogeny invariant.
Therefore in the database, when we say that an isogeny class is principally polarizable, we mean that there exists some abelian variety in the class which is principally polarizable.

Of the 2,945,722 isogeny classes in the database, there are only 3037 such that we can not determine if the class has a principal polarization.
These isogeny classes all occur in dimension greater than 3: 358 are in dimension 4, 515 are in dimension 5, and 2164 are in dimension 6. 

We now discuss methods for determining when an isogeny class is principally polarizable. 
This code was provided by Howe; the version we use is implemented in the \texttt{has\_principally\_polarizable} function \cite{abvarfq}.

For simplicity, the algorithm considers only the case of simple abelian varieties. We note that since an isogeny class is principally polarizable when all of its isogeny factors are, the database does indicate that  a nonsimple isogeny class is principally polarizable when all of its factors are. We remind the reader however that, as remarked above, the converse is not true, which does mean that certain non-simple cases are currently completely out of the reach of our implemented tests, and might be principally polarizable even though not all of their factors are. 

Our algorithm depends on the dimension of the isogeny class:
In the case $g=1$, every isogeny class is principally polarizable.
In the case $g=2$, we apply a result of \cite{Howe2008}, which tests for principal polarizations based on a condition on the coefficients of the Weil polynomial. 
These conditions are that an isogeny class is \emph{not} principally polarizable if and only if $a_1^2 - a_2 = q$, $a_2<0$, and every prime divisor of $a_2$ is $1 \mod 3$.

We now list the collection of tests that are implemented in the database in higher dimension. First, if $g$ is odd and $A$ is simple then its isogeny class is principally polarizable, which takes care of $g = 3$ and $g=5$.

In addition, in the ordinary case, we can completely determine whether an isogeny class is principally polarizable, using Corollary 11.4 and Proposition 11.5 of \cite{Howe1995}.
We note that the details of the proofs are given in \S14 of \cite{Howe1993} for readers who would like to see them.
The test relies on the fact that in the ordinary case, $N_{K/\QQ} (\pi - q/\pi)$ is always a square and we can consider its positive square root $N$.
If $q > 2$, there is a principally polarized variety in the isogeny class if and only if $N \equiv (\mbox{coeff of $T^g$}) \mod q$.
If $q = 2$, we have a similar congruence condition but for a power of $2$:
there is a principally polarized abelian variety in the isogeny class if and only if $N \equiv (\mbox{coeff of $T^g$}) \mod 2^2$. 

\begin{remark}
	In \cite{Howe1995}, Howe gives in fact algorithm for determining when an ordinary isogeny class (simple or not) contains a principally polarizable variety, but the non-simple case has not been implemented yet and we did not implement it for the database.
\end{remark}

\begin{remark}
Howe has explained to us that the congruence conditions come from wanting to determine if $N$, the positive square root of $N_{K/\QQ}(\pi+q/\pi)$, is the same as the square root specified by certain $p$-adic conditions. 
The result is then explained by the fact that the middle coefficient of the Weil polynomial is congruent modulo $q$ to the square root specified by the $p$-adic conditions. 
When $q>2$, the congruence is enough to compare the signs of the two square roots.
However, when $q=2$ we have $1 \equiv -1 \mod 2$, which explains the need to work modulo $4$. 
\end{remark}

Finally, in even dimension, we verify some conditions on the field $K$ generated by the Weil $q$-number $\pi$ to allow us to determine if certain isogeny classes are principally polarizable (this is \cite[Theorem~1.1]{Howe1996}). 
First, if $K$ is totally real then the class is principally polarizable. (Although if $g \geq 4$ and $A$ is simple, $K$ is always a CM field.) 
Otherwise, $K$ is a CM field. 
To set up the notation we will need, let $P(T)$ be the characteristic polynomial of an isogeny class of abelian varieties.
Write $P(T)=h(T)^e$ for $h$ irreducible; then $h$ is the minimal polynomial of a Weil $q$-number $\pi$. 
	Let $K = \QQ(\pi)$ and $K_+=\QQ(\pi+q/\pi)$ be the CM field and its totally real subfield, respectively. 
	Then we know that there is a principally polarizable abelian variety in the isogeny class associated to $P(T)$ if either  $K/K_+$ is ramified at a finite prime, or there is a prime of $K_+$ that divides $\pi - q/\pi$ and is inert in $K/K_+$. 
	We note that this second condition requires some work to test for, and we use the fact that a prime is inert in $K/K_+$ if and only if the prime ideal is equal to its complex conjugate to do so.

\subsection{Jacobian Testing}

Given a characteristic polynomial, to test if there exists a Jacobian with this characteristic polynomial, we apply six results from Howe and Lauter's Magma package \url{http://ewhowe.com/Magma/IsogenyClasses.magma}. 
(They were re-implemented for the LMFDB by Howe.)
This code accompanies the paper \cite{Howe2012}; see especially Section 6 of the article for a high-level overview of the software. 
In addition, the comments in the code are excellent, so rather than repeat an explanation of the tests, we refer the reader to these two excellent references. 
We note that currently the LMFDB does not implement positive Jacobian testing in genus 4 and higher. 

\subsection{Angle Rank} \label{ssec:angle rank3}
We have two algorithms to compute the angle rank of an isogeny class of abelian varieties: one that is numerical, and one that is algebraic and therefore yields a provably true answer. 
In the current version of the LMFDB, the angle rank $\delta(A)$ is computed numerically using lattice basis reduction.

This is done in the following way: We first approximate the roots $\alpha_i$ of the characteristic polynomial $P(T)$ numerically as pairs of floats (the real and imaginary parts of the number). 
We then pair each root $\alpha_i$ to its complex conjugate (which is also a root of $P(T)$) and retain only one number from each pair of complex conjugates, as we know that $\arg(\alpha_i)/2\pi$ and $\arg(\overline{\alpha}_i)/2\pi$ have a linear relation over $\mathbf{Q}$.
Following this, we compute $\arg(\alpha_i)/2\pi$ numerically (using the principal branch of the logarithm) for the $g$ remaining $\alpha_i$ to get values $t_1,\ldots, t_g$. 
Finally, to determine the dimension of the $\mathbf{Q}$-span of the appropriate set (the reader can go to \S \ref{ssec:angle rank2} for a reminder of the definition of angle rank), we apply an LLL algorithm with a certain precision to the tuple $[t_1,\ldots, t_g, 1]$ (\texttt{PARI}'s \texttt{lindep}).
Note that in these computations the branch of the logarithm chosen for the computations does not matter because we eventually take a $\QQ$-vector space span with $1$.

For the sake of completeness, we present below the algebraic algorithm yielding a provably true answer as well.
Roughly speaking, the algorithm relies on expressing the roots of $P(T)$ in a common generating set for some $S$-units of the splitting field of $P(T)$.
It would of course have been preferable to use this algorithm in the database instead, but at the time this computation was performed for the database (2015), the implementation of the software computing $S$-units in \texttt{Sage} was not fast enough to deploy on the full database. See Subsection~\ref{ssec:bottlenecks} for a possible workaround.

We now present the algebraic algorithm: Let $A$ be an abelian variety over $\FF_q$ where $q = p^a$.
Let $K$ be the splitting field of $P(T)=P_A(T) = \prod_{i=1}^{2g} (T-\alpha_i)$.
For the computation of the angle rank we consider 
$$ \Gamma = \langle \alpha_1, \alpha_2, \ldots,\alpha_{2g} \rangle, $$ 
the multiplicative subgroup of $K^{\times}$ generated by the roots of $P_A(T)$. 
In order to compute $\delta(A)$ we use the fact that 
$$\rk(\Gamma) = \delta(A).$$
The first two authors discuss the group $\Gamma$ in the sister paper \cite{Dupuy2020b}, where some explicit relations are computed. 
Let $S = \lbrace P \in \Spec(\OO_K): P \vert p \rbrace $ be the collection of primes of $K$ above $p$. 
The key observation is that $\Gamma$ is a subgroup of the group of $S$-units of $K$:
$$\Gamma \leq \OO_{K,S}^{\times}.$$
Using \texttt{Sage} we can then compute generators\footnote{The particular function we use was written by Cremona: \url{http://doc.sagemath.org/html/en/reference/number_fields/sage/rings/number_field/unit_group.html}.
We also remark that there is an additional quite large bottleneck in this algorithm due to the need to compute the splitting field of the characteristic polynomial.  } for $\OO_{K,S}^{\times}$:
 $$\OO_{K,S}^{\times} = \langle \zeta, u_1,\ldots, u_r \rangle$$
where $\zeta$ is a root of unity generating the torsion part of $\OO_{K,S}^{\times}$ and $u_1,\ldots,u_r$ are of infinite order, and attempt to compute the rank of $\Gamma$ in this basis. 
Before proceeding, however, we eliminate the torsion part of $\Gamma$, if any:
Let $m$ be the cardinality of the group of roots of unity in $K^{\times}$.
Then to compute the rank of $\Gamma$ it suffices in fact to compute the rank of $\Gamma^m = \langle \alpha_1^m,\ldots, \alpha_{2g}^m \rangle$, since  
\begin{equation*}
\rk(\Gamma) = \delta(A_{\FF_q}) = \delta(A_{\FF_{q^m}}) = \rk(\Gamma^m).
\end{equation*} 
Here, the equality of angle ranks follows from the fact that $\delta$ computes $\QQ$-linear relations, and $\arg(\alpha)/2\pi$ and $\arg(\beta)/2\pi$ have a $\QQ$-linear relation if and only if $\arg(\alpha^m)/2\pi$ and $\arg(\beta^m)/2\pi$ do.

We then write each $\alpha_i^m$ in terms of the generating elements $u_j$, and obtain a vector of exponents with integer coefficients. 
We then form a $r\times 2g$ matrix $Y$ whose columns are these vectors, and the rank of this matrix gives us our answer:
 $$ \rk(Y) = \delta(A_{\FF_q}) +1. $$
We note that the explicit relations in $\Gamma$ mentioned above are derived from the matrix $Y$;
several examples of these computations can be found in \cite{Dupuy2020b}.

\section{Statistics vs.\ Heuristics}
\label{S:statistics}

This database naturally invites investigation of the following motivating questions:
\begin{itemize}
\item
How does the number of objects in the database vary with $q$ and $g$?
\item
How about if we sort by the Galois group of $\QQ(\pi)$, the Newton polygon, or the angle rank?
\item
How is the number of points on the abelian variety distributed?
\item
What are the extreme values of the number of points?
\end{itemize}

In this section, we gather the data available  in the database and informing these questions, and compare it to the predictions given by heuristics.

\subsection{The Number of Isogeny Classes}
\label{subsec:number of isogeny classes}

In order to set bounding boxes for our data collection, we needed to estimate
$$N(g,q) = \mbox{number of isogeny classes of $g$-dimensional abelian varieties over $\Fq$}. $$ 
We initially chose the limits described in Table~\ref{T:number-of-classes} using an incorrect
estimate on the growth of $N(g,q)$: Believing that $N(g,q)$ grows like $q^{g(g+1)/2}$ (rather than the correct asymptotic $q^{g(g+1)/4}$; see below), for each $g$ we included data for $q$ with $q^{g(g+1)/2} \le 10^7$ to bound the number of isogeny classes per pair $(g,q)$.\footnote{We also included $(5, 3)$, where $q^{g(g+1)/2} \approx 1.4 \cdot 10^7$, and we only included $q$ up to $500$ for $g=1$, rather than $10^7$.}
We later extended the bounds in order to include more fields of small characteristic in dimension up to $3$,
and to make the range of included $q$ contiguous in dimension $3$.\footnote{Namely, we added powers of $2, 3, 5, 7$ up to $1024$ for $g \in \{1, 2\}$, and raised the bound for $g=3$ from $13$ to $25$.}

We now give a more careful analysis that better models the numerical data that we have observed for $N(g,q)$.
Since by Honda-Tate these isogeny classes are in bijection with characteristic polynomials, one reasonable heuristic is to count Weil $q$-polynomials of degree $2g$. This amounts to counting 
lattice points in the set of $(a_1,\dots,a_g) \in \RR^g$ for which the polynomial
\[
T^{g} + a_1 T^{g-1} + \cdots + a_g
\]
has all roots in the interval $[-2\sqrt{q}, 2\sqrt{q}]$; it is reasonable to approximate this count by the volume of the region. This volume has been computed by DiPippo and Howe  \cite[Proposition~2.2.1]{DiPippo1998} \footnote{We remark that the bounds used in \emph{op.\ cit}.\ are also related to the original iterator described in \cite{Kedlaya2008} which some reader might find enlightening.}: it equals
\begin{equation} \label{eq:DiPippoHowe}
\left( \frac{2^g}{g!} \prod_{i=1}^g \left( \frac{2i}{2i-1} \right)^{g+1-i}\right) q^{g(g+1)/4}.
\end{equation}
This turns out to be a good prediction in practice; see Table~\ref{T:volume-prediction} for some examples.

For the sake of completeness, we note two facts: First, according to \cite{DiPippo1998}, for $q$ large compared to $g$, the dominant contribution to the count is from ordinary abelian varieties. 
Secondly, to obtain the number of ordinary isogeny classes from the formula for the total number of isogeny classes, we simply multiply by a factor of $\varphi(q)/q$.

\begin{table}
\begin{center}
\begin{tabular}{cc|cccc}
& & \multicolumn{2}{c}{Ordinary} & \multicolumn{2}{c}{Arbitrary} \\
$g$ & $q$ & Predicted & Actual & Predicted & Actual  \\
\hline
3 & 25 & 284444 & 284740 & 355556 & 332166 \\
4 & 5 & 104025 & 105600 & 130032 & 132839 \\
5 & 3 & 170796 & 171180 & 256194 & 267465 \\
6 & 2 & 72362 & 74122 & 144724 & 164937\\
\end{tabular}	
\end{center}
\caption{\label{T:volume-prediction} Predicted versus actual values for the number of isogeny classes of ordinary/arbitrary abelian varieties of dimension $g$ over $\Fq$.}
\end{table}

Another way to evaluate our prediction for $N(g,q)$ is to make a least-squares fit for values $a,b$ such that
$\log N(g,q) \approx a \log q + b.$ The results are given in Table~\ref{T:least-square-values}.
The log-log plot of $q$ vs $N(g,q)$ is given in Figure~\ref{F:least-squares}.

\begin{table}
\begin{center}
\begin{tabular}{c|cccc}
& \multicolumn{2}{c}{$a$} & \multicolumn{2}{c}{$b$} \\
$g$ & Predicted & Actual & Predicted & Actual  \\
\hline 	1 & 0.5 & 0.4971 & 1.3863 & 1.3717 \\
2 & 1.5 & 1.4178 & 2.3671 &  2.5302 \\
3 & 3 &  2.9135 & 3.1248 &  3.2598\\
4 & 5 & 4.5452 & 3.7283 & 4.2707 \\
5 & 7.5 & 7.2188 & 4.2141 & 4.5660  \\
\end{tabular}	
\end{center}
\caption{\label{T:least-square-values} Predicted versus actual (least squares) values for the equation $\log N(g,q) = a \log(q) + b $.}
\end{table}

\begin{figure}
	\begin{center}
	\includegraphics[scale=0.75]{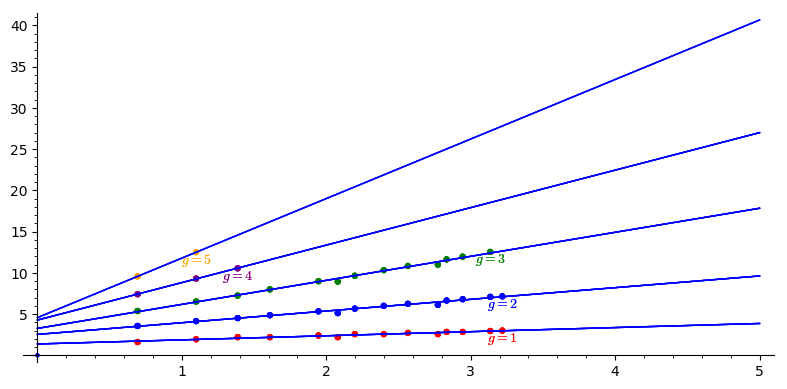}
	\end{center}
	\caption{\label{F:least-squares} Plots of $(\log q, \log N(g,q))$ for each $g$ and $q \le 27$ in the database. 
	The least square values of $\log N(g,q) = a \log q + b$ are given in Table~\ref{T:least-square-values}.  }
\end{figure}

While we do not yet have data for isomorphism classes (see \S\ref{ssec:isormorphism classes}),
it is worth noting that one can derive heuristics for counts of isomorphism classes within an isogeny class
using Oort's theory of \emph{isogeny leaves} in moduli spaces of abelian varieties \cite{Oort2009}.

\subsection{Galois Groups}  \label{subsec:Galois groups}

Let $P(T)$ be the characteristic polynomial of an isogeny class of abelian varieties, and $\QQ(\pi)^{\gal}$ be the splitting field of $P$ over $\mathbf{Q}$.
How can one expect $\Gal(\QQ(\pi)^{\gal}/\QQ)$ to be distributed as the isogeny class varies?
To answer this question, we need to explain Malle's conjecture \cites{Malle2002, Malle2004}
and the invariant $a(G)$ for a finite group $G \subseteq S_n$. 
The exposition given here follows the treatment in \cite{Malle2002} (see the introduction for example).
Given $g \in S_n$ we define 
 $$ \ind(g) = n - \# \mbox{number of cycles in the cycle decomposition of $g$.} $$
We then define 
 $$ a(G) = \frac{1}{\min \lbrace \ind(g): g \in G\setminus 1 \rbrace  }. $$

\begin{example}
	\begin{enumerate}
		\item $\ind( (12)(34) ) = 4 - 2 = 2$
		\item $a(S_4) = 1$ 
		\item $a(D_4) = 1$
		\item $a(A_4) = 1/2$
		\item $a(C_4) = 1/2$
	\end{enumerate}
\end{example}

Let $K$ be a number field, and $n$ be an integer.
We define the following counting functions for number fields $L$ of degree $n$ over $K$: 
\begin{gather*}
N_{K,n}(X) = \#\lbrace L : [L:K]=n, \vert \Disc(L/K) \vert \leq X\rbrace,\\
N_{K,n,G}(X) = \#\lbrace L : [L:K]=n, \vert \Disc(L/K) \vert \leq X, \Gal(L^{\gal}/K) = G \rbrace. 
\end{gather*}
 
First, we have Linnik's conjecture  (c.f. \emph{op.\ cit}.)
that there exists a constant $c_0=c_0(K,n)$ such that  
\begin{equation}\label{E:discriminant}
N_{K,n}(X) \sim c_0 X \mbox{ as $X\to \infty$ }. 
\end{equation}
Next, the (weak) Malle conjecture states that for all $\varepsilon>0$ there exist constants $c_1, c_{2,\varepsilon}$ such that 
\begin{equation}\label{E:malle}
  c_1 X^{a(G)} < N_{K,n,G}(X) < c_{2,\varepsilon} X^{a(G)+\varepsilon}. 
\end{equation}

Assuming these conjectures, given a base field $K$, we can describe how the proportion of extensions with Galois group $G$ should behave on a log-log scale:
\begin{lemma}
	Assuming equations \eqref{E:discriminant}  and \eqref{E:malle}  we have
	 $$ \lim_{X\to \infty} \frac{\log N_{K,G,n}(X)}{\log N_{K,n}(X)} = a(G). $$
\end{lemma}
\begin{proof}
	Assume the Malle and Linnik conjectures. 
	We will prove the upper bound and omit the proof of the lower bound as it is similar and easier. 
	Since $N_{K,n}(X) \sim c_0X$, there exists some $\varepsilon_1(X)$, approaching zero as $X\to \infty$, such that $N_{K,n}(X) = c_0X(1+\varepsilon_1(X))$. 
	Applying Malle's conjecture, we then have, for all $\varepsilon>0$,
	\[
	\frac{\log N_{K,n,G}(X) }{\log N_{K,n}(X)} < \frac{ a(G)+\varepsilon + \frac{\log c_{2,\varepsilon}}{\log(X)} }{1+ \frac{\log c_0}{\log(X)}+\frac{\log(1+\varepsilon_1(X))}{\log(X)}} \to a(G) + \varepsilon \mbox{\ \ \  as $X\to \infty$. }
	\]
	This proves that for all $\varepsilon>0$
	\[
	\lim_{X \to \infty} \frac{\log N_{K,n,G}(X) }{\log N_{K,n}(X)} < a(G) + \varepsilon,
	\]
	and hence 
	\[
	\lim_{X \to \infty} \frac{\log N_{K,n,G}(X) }{\log N_{K,n}(X)} \leq a(G).
	\]
	A simpler bound using the lower end of Malle's conjecture proves 
	\[
	a(G) \le \lim_{X \to \infty} \frac{\log N_{K,n,G}(X) }{\log N_{K,n}(X)},
	\]
	which gives the result.
\end{proof}

Because of this result, we might expect that a group $G$ should appear as $\Gal(\QQ(\pi)^{\gal}/\QQ)$ with frequency such that the log-log ratios have a limiting value for each $G$ as $q \to \infty$.
While this does seem to be the case, the precise value of these ratios does not seem to coincide with Malle's constant $a(G)$, and seems to be more complicated.
Table~\ref{T:malle} shows them for $g=3$.
If we let $\widetilde{a}(G)$ denote these limits, to the authors it seems that
$$\widetilde{a}(\texttt{6T11}) = 1, \quad \widetilde{a}(\texttt{6T3}) = \widetilde{a}(\texttt{6T6}) \approx 3/5, \quad 1/10\leq \widetilde{a}(\texttt{6T1}) \leq 3/10, $$
and we have no theoretical explanation for why this is the case.
We do note that there seems to be some partial progress on these type of conditional Malle distributions in \cite{BarqueroSanchez2017}.

\begin{remark}
Alternatively, it may be the case that these distributions fit closer to a ``polynomials-in-a-box'' distribution. 
The study of the Galois group of a random polynomial goes back to van der Waerden \cite{Waerden1936}, and an excellent summary (and further developments) is given in \cite{Zywina2010}.
One can prove (see \cite{Bary-Soroker2020}) that the number of monic polynomials of degree $d$ which are irreducible and have coefficients in $[-L/2,L/2]$ is asymptotic to $L^{d}$. 
Let's call this number $B_d(L)$.
If we let $B_{d,G}(L)$ be the number of monic irreducible polynomials with coefficients in $[-L/2,L/2]$ with Galois group $G$, then in analogy with Malle one could naively guess that for every group $G$ transitive on $d$ elements,  there exists some $\alpha(G)\leq 1$ such that for every $\varepsilon>0$, there exist constants $c_1$ and $c_{2,\varepsilon}$ and $R>0$ such that for every $L>R$ one has 
$$c_1 L^{d\alpha(G)}  <  B_{d,G}(L)  < c_{2,\varepsilon} L^{d(\alpha(G)+\varepsilon)}.$$
If such constants exist, is it the case that $\widetilde{a}(G) = \alpha(G)$? 
\end{remark}

\begin{table}
	\begin{center}
		\begin{tabular}{l|llllllllllllll}
			Group$\backslash q$ & $2$ & $3$ & $4$ & $5$ & $7$ & $8$ & $9$ & $11$ & $13$  \\ \hline
			6T1 & $0.31636$ & $0.11844$ & $0.26715$ & $-\infty$ & $0.30254$ & $0.23744$ & $0.14698$ & $0.25954$ & $0.28031$ \\
			2T1 & $-\infty$ & $-\infty$ & $-\infty$ & $-\infty$ & $-\infty$ & $0.079146$ & $-\infty$ & $-\infty$ & $-\infty$ \\
			6T3 & $0.60225$ & $0.59221$ & $0.66064$ & $0.58783$ & $0.60854$ & $0.54666$ & $0.63575$ & $0.57555$ & $0.60162$ \\
			6T11 & $0.88343$ & $0.96285$ & $0.97157$ & $0.98768$ & $0.99256$ & $0.99380$ & $0.99424$ & $0.99707$ & $0.99743$ \\
			6T6 & $0.60225$ & $0.60257$ & $0.58913$ & $0.59649$ & $0.58877$ & $0.60836$ & $0.58708$ & $0.59022$ & $0.59007$ \\
		\end{tabular}
	\end{center}
\caption{\label{T:malle}For $g=3$ we have plotted $\log N(g,q,G)/\log N(g,q)$. 
Observe the phenomena for $\mathtt{6T1}$ at powers of $3$. 
Also, note that none of these can possibly be $a(G)$ for some $G$, which must be of the form $1/n$ for $1\leq n \leq 2g=6$ (as these number are 1.00, 0.50, 0.33, 0.25, 0.20, 0.16).}
\end{table}

\subsection{Newton Polygons Data and \texorpdfstring{$p$}{p}-rank Strata}

As discussed in \S\ref{subsec:Newton polygons}, for any given positive integer $d$, the coarse moduli space
$\A_{g,d}$ of $g$-dimensional abelian varieties equipped with a polarization of degree $d^2$ admits a locally closed stratification by Newton polygons, in which the stratum corresponding to an individual polygon is equidimensional of codimension equal to the elevation of the polygon.
A reasonable guess is that for any given eligible Newton polygon $P$ in dimension $g$,
the proportion of isogeny classes of abelian varieties over $\FF_q$ with Newton polygons lying on or above $P$
is $cq^{-e}$ where $e$ is the elevation of $P$ and $c$ is the number of geometrically irreducible 
components of the stratum over $\FF_q$. By Theorem~\ref{T:NP stratification}, $c=1$ unless $P$ is the supersingular stratum; 
in the supersingular case the stratum is reducible and not all irreducible components may be defined over $\FF_q$, but it is guaranteed that $c > 0$ \cite{Yu2017}. 
One can even give an explicit formula for $c$ when $q=p$ \cite{Ibukiyama2018}*{Thm. 4.6}.

For example, in dimension 3 the Newton polygons are linearly ordered (see Figure~\ref{F:newton}).
In Figure~\ref{F:newton-numbers}, we give a plot of 
\[
\log_q\left(\frac{N(3, q, P)}{N(3, q)}\right)
\]
for each of the five possible Newton polygons, where $N(g, q, P)$ is the number of isogeny classes with Newton polygon on or above $P$, and $N(g, q)$ is the total number of isogeny classes.
Note that the values for $q$ prime agree quite well with the discussion above, while for non-prime $q$ there are more isogeny classes in the smaller strata than expected.
We have no explanation for this behavior. 
Moreover, the supersingular stratum lies consistently above $-4$ and $\log_q\left(N(3, q, P) / N(3, q)\right)$ is increasing with $q$, suggesting that the number of geometrically irreducible components increases as some nonzero power of $q$.

\begin{figure}[!htbp]
\begin{center}
	\includegraphics[scale=0.5]{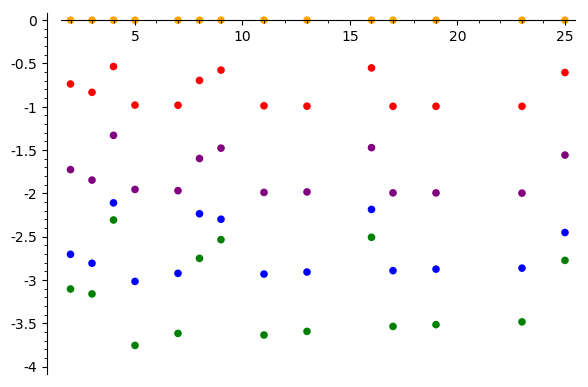}
\end{center}
\caption{$\log_q(N(3, q, P) / N(3, q))$ as a function of $q$, where $N(g, q, P)$ is the number of isogeny classes with Newton polygon on or above $P$.  Colors indicate $P$, with orange the ordinary stratum and green the supersingular stratum. \label{F:newton-numbers}}
\end{figure}
	
\begin{remark}
In dimension 4, every Newton polygon stratum in $\A_{4,1}$ occurs for the Jacobian of some curve. For strata of codimension at least $2$ this is implied by \cite{Shankar2018}*{Proposition 2.4}; this leaves only the ordinary stratum which includes the generic Jacobian, and the almost ordinary stratum which includes the generic
non-ordinary Jacobian.
\end{remark}

\begin{problem} \label{problem: Pries p-ranks}
The following question is due to Pries (see \url{http://aimpl.org/cohomabelian/}):
Let $p$ be a prime number. For a fixed dimension $g$ and a fixed $p$-rank $f$, what is the smallest field of definition of a simple abelian variety over $\overline{\FF}_p$ with dimension $g$ and $p$-rank $f$?
\end{problem}

For each pair $(g,q)$ for which the LMFDB contains data, we have
verified that each of $0,\dots,g$ occurs as the $p$-rank of at least one simple abelian variety of dimension $g$ over $\FF_q$.

\subsection{Frobenius Angle Rank} \label{subsec:angle rank}
Table~\ref{T:the-dip} shows for each $g$ and angle rank $\delta$ with $0\leq \delta \leq g$ which $p$-ranks do not appear in our database. 
\begin{longtable}{lll}
	Dimension & Angle rank & Forbidden $p$-ranks \\ \hline
	\endfirsthead
	Dimension & Angle rank & Forbidden $p$-ranks \\ \hline
	\endhead
	 1 & 0 & 1  \\
	    & 1 & 0  \\[0.2cm]
	 2 & 0 & 1,2  \\
	    & 1 & 0,1  \\
	    & 2 & 0  \\[0.2cm]
	 3 & 0 & 1,2,3  \\
	    & 1 & 1,2  \\
	    & 2 & 0,1,3 \\
	    & 3 &  none \\[0.2cm]
	 4  & 0 & 1,2,3,4  \\
	    & 1 & 1,2,3  \\
	    & 2 & 1,3  \\
	    & 3 & 3  \\
	    & 4 &  none \\[0.2cm]
	 5  & 0 & 0,1,2,3,4,5 \\
	    & 1 & 1,2,3,4 \\
	    & 2 &0,1,2,3,4,5  \\
	     & 3 & 0,1,2,3,4,5\\
	    & 4 & 1,3,5  \\
	     & 5 & none \\[0.2cm]
	   6 & 0 & 1,2,3,4,5,6  \\
	     & 1 & 1,2,3,4,5  \\
	     & 2 & 1,5 \\
	     & 3 & 1,3,5 \\
	     & 4 & 0,1,3,5,6 \\
	     & 5 & 5  \\
	     & 6 & none  \\
	\caption{Excluded $p$-ranks for simple abelian varieties. \label{T:the-dip} }\\
\end{longtable}
From Table \ref{T:the-dip} some simple patterns emerge.
\begin{itemize}
\item 
Angle rank 0 implies $p$-rank 0. This is known in general because angle rank 0 is equivalent to supersingularity.
\item
Up to $g=4$, every positive angle rank can occur for ordinary abelian varieties (i.e., for $p$-rank $g$).
However, please note that for $g=5$ we only cover $q=2,3$, and for $g=6$ we only cover $q=2$.
\item
If $\delta = 1$ then the $p$-rank is 0 or $g$. 
\item
For $g \geq 3$, if $\delta = 2$ then the $p$-rank cannot be $1$.
\item
The $p$-rank $g-1$ only occurs when $\delta = g-1$ or $\delta = g$. This follows from a theorem of 
Lenstra--Zarhin \cite[Theorem~5.7]{Lenstra1993}, who also show that the case $\delta = g-1$ cannot occur if $g$ is even \cite[Theorem~5.8]{Lenstra1993}.
\end{itemize}
For  discussions in this direction we refer the reader to the forthcoming paper \cite{Dupuy2020b}. 

\subsection{Endomorphism Algebras} \label{subsec:endomorphism algebras}

We motivate this section with the following open problem:

\begin{problem}[{\cite[Open Problem 22.6]{Oort2008}}]
	For each $g>0$ and $p$ determine the possible endomorphism algebras occurring in that characteristic. 
\end{problem}

As a first step we focus on simple abelian varieties, since the general case is expressible
in terms of matrix algebras over the endomorphism algebras of the simple constituents.
We can then break this problem up into two parts: understanding the center of the endomorphism algebra,
and understanding the endomorphism algebra as a division algebra over its center.

The center of the endomorphism algebra of an abelian variety $A$ is just
the number field $\QQ(\pi)$ defined by $h_A(T)$, where $P_A(T) = h_A(T)^e$.
In order to analyze the possible centers that can arise, we look at statistics on the discriminant $\Delta$ of $\QQ(\pi)$.
It is useful to normalize the discriminant in two ways.
First, we consider the root discriminant $\rd = \lvert \Delta \rvert^{1/n}$ instead, where $n = [\QQ(\pi):\QQ]$.
The root discriminant is often more useful when considering number fields of different degrees.
Second, we use the following result to rescale the root discriminant in a way that allows comparison
across different values of $q$.

\begin{theorem}
Let $A$ be a simple abelian variety of dimension $g$ over $\Fq$ with associated Weil number $\pi$.
Then the root discriminant $\rd$ of $\QQ(\pi)$ is bounded by
\[
\rd \le 2gq^{g/2}.
\]
\end{theorem}
\begin{proof}
By \cite{Mordell1960}*{Eq. 17}, the maximum possible \emph{polynomial} discriminant for a polynomial
of degree $2g$ whose roots all have absolute value $\sqrt{q}$ is $(2g)^{2g} q^{g(2g-1)}$. 
Applied to the polynomial $h_A(T)$, this yields the bound
\[
\rd \leq \frac{2g}{e} q^{g/e-1/2}.
\]
This implies the desired bound when $e>1$, so we may assume hereafter that $e=1$.
To improve the bound in this case, we distinguish between the discriminant $\Delta$ of $\QQ(\pi)$
and the discriminant $\Delta'$ of $P_A(T)$; it will suffice to check that the ratio $\Delta'/\Delta$ is divisible by $q^{g(g-1)}$. 

Put $\beta = \pi + q/\pi$, let $\Delta_0$ be the discriminant of $\QQ(\beta)$, and let 
$\Delta'_0$ be the discriminant of the minimal polynomial of $\beta$ over $\QQ$. 
Let $\alpha_1,\dots,\alpha_{2g}$ be the conjugates of $\pi$ in $\QQ^{\alg}$, 
sorted so that $\alpha_i \alpha_{2g-i} = q$ for $i=1,\dots,g$. 
Then on one hand,
\begin{align*}
\frac{\Delta'}{(\Delta'_0)^2} &= 
\left( \prod_{i=1}^g (\alpha_i - \alpha_{2g-i})^2  \right)
\left(
\prod_{i=1}^g \prod_{j=i+1}^g \left( \frac{(\alpha_i - \alpha_j)(\alpha_i - \alpha_{2g-j})(\alpha_j - \alpha_{2g-i})(\alpha_{2g-j}-\alpha_{2g-i})}{(\alpha_i + \alpha_{2g-i} - \alpha_j - \alpha_{2g-j})^2}
\right)^2
\right) \\
&= \left( \prod_{i=1}^g (\alpha_i - q/\alpha_i)^2 \right)
\left( 
\prod_{i=1}^g \prod_{j=i+1}^g \left( 
\frac{(\alpha_i - \alpha_j)(\alpha_i - q/\alpha_j)(\alpha_j - q/\alpha_i)(q/\alpha_j-q/\alpha_i)}{(\alpha_i + q/\alpha_i - \alpha_j - q/\alpha_j)^2}
\right)^2
\right) \\
&= \left( \prod_{i=1}^g (\alpha_i - q/\alpha_i)^2 \right)
\left( 
\prod_{i=1}^g \prod_{j=i+1}^g \left( 
\frac{\alpha_i^{-2} \alpha_j^{-2} (\alpha_i - \alpha_j)(\alpha_i \alpha_j - q)(\alpha_i \alpha_j - q)(q \alpha_i - q \alpha_j)}{\alpha_i^{-2} \alpha_j^{-2} (\alpha_i - \alpha_j)^2 (\alpha_i \alpha_j - q)^2}
\right)^2
\right) \\
&= q^{g(g-1)} \prod_{i=1}^g (\alpha_i - q/\alpha_i)^2.
\end{align*}

On the other hand, the relative discriminant of $\QQ(\pi)$ over $\QQ(\beta)$ divides 
the polynomial discriminant of $x^2 - \beta x + q$, which is $(\pi - q/\pi)^2$.
Consequently, 
\[
\frac{\Delta}{\Delta_0^2} \qquad \mbox{divides} \qquad \prod_{i=1}^g (\alpha_i - q/\alpha_i)^2.
\]
By writing
\[
\frac{\Delta'}{\Delta} = \frac{\Delta'}{(\Delta'_0)^2} \left( \frac{\Delta'_0}{\Delta_0} \right)^2
\frac{\Delta_0^2}{\Delta}
\]
and noting that $\Delta'_0$ is divisible by $\Delta_0$, we deduce that $\Delta'/\Delta$ is divisible by $q^{g(g-1)}$ as claimed.
\end{proof}

In Figures \ref{F:disc-hist-1}--\ref{F:disc-hist-6} we give distributions of both the polynomial and number field root discriminants for the different values of $g$ in the database.  In the polynomial case, we divide the root discriminant by $2gq^{\frac{2g-1}2}$, and in the number field case we divide by $2gq^{g/2}$.  The distribution in the number field case appears to be a sum of copies of the polynomial distribution after further rescaling by the appropriate roots of reciprocals of integers.  This phenomenon is especially apparent in the $g=2$ case because of the simple nature of the polynomial distribution.  Large spikes can be seen at $1/2$, $1/3$, $1/4$ and $1/5$ corresponding to an extra factor of $2^4$, $3^4$, $4^4$ or $5^4$ in the index of the maximal order in the equation order of the number field; smaller spikes are visible at $1/\sqrt{2}$ and $1/\sqrt{3}$ corresponding to extra factors of $2^2$ and $3^2$.

Of course, while the endomorphism algebra is usually commutative, sometimes it is not.
In order to give some insight into the non-commutative cases, we provide statistics on the possible
Brauer invariants of the endomorphism algebra as a division algebra over its center.
Table \ref{table:brauer_inv} in the center summarizes the results.  The length of the sequence of
invariants gives the number of places above $p$ in the center, and we have collapsed
all of the commutative endomorphism algebras into a single row for each value of $g$.

\subsection{Isogeny Sato-Tate distribution} \label{subsec:sato-aint}

What is the distribution of $\# A(\FF_q)$? What about when we restrict to ones with certain invariant types?
From the Lang-Weil estimates (\S\ref{ssec:point_count_bounds}) we know that 
\[
\#A(\FF_q) = q^g + O(q^{g-1/2}) \mbox{ as $q \to \infty$}.
\]
This asymptotic suggests that the normalized error
\[
E := \frac{\#A(\FF_q) -q^g}{q^{g-1/2}}
\]
will form an interesting probability distribution $P_{g,q}$ as we vary over all $A$'s of a fixed dimension $g$ defined over $\FF_q$. 

Let us consider what happens if we fix $g$ and take a limit as $q \to \infty$. Writing $\alpha_1,\dots,\alpha_{2g}$ for the Frobenius eigenvalues with $\alpha_i \alpha_{g+i} = q$ for $i=1,\dots,g$, we have
\[
E = q^{-g+1/2} \left(\prod_{i=1}^{2g} (1 - \alpha_i) - \prod_{i=1}^{2g} \alpha_i \right)
= \sum_{i=1}^{2g} q^{-1/2} \alpha_i + o(q^{-1/2});
\]
consequently, the distribution of $E$ will have the same limiting behavior as the distribution of
the normalized Frobenius trace of $A$.

The philosophy of Katz--Sarnak
\cite{Katz1999b} would predict that the distribution of the Frobenius trace should converge to the trace distribution for random matrices in the Lie group $\operatorname{USp}(2g)$. This convergence holds if we average over isomorphism classes of principally polarized abelian varieties, as this forms a geometric family with maximal monodromy \cite[Theorem~11.0.4]{Katz1999b} to which one may apply Deligne's equidistribution theorem \cite[Theorem~9.2.6]{Katz1999b}.

However, since we do not currently have the data of how many isomorphism classes constitute a given isogeny class, 
we are only able to compute the average over isogeny classes. We thus predict a different distribution,
given by a function whose value at $a_1$ is proportional to the measure of the set of $(a_2,\dots,a_g) \in \RR^{g-1}$ for which $T^{g} + a_1 T^{g-1} + \cdots + a_g$ has all roots in $[-2, 2]$. 
We compute this distribution using the method of DiPippo--Howe (see \S\ref{subsec:number of isogeny classes}).
By computing Jacobian determinants, we see that integrating 1 over the space of coefficients $(a_1,\dots,a_g)$ is the same as integrating $1/g!$ over the space of power sums $(p_1,\dots,p_g)$, or integrating $1/g!$ times the Vandermonde determinant $V(r_1,\dots,r_g) = \prod_{1 \leq i<j\leq g} (r_j-r_i)$ over the space of ordered tuples of roots $(r_1,\dots,r_g)$. That is, the desired distribution is given (up to a normalizing factor) by the distribution function
\[
f(x) = \int_{S \cap H_x} V(r_1,\dots,r_g) d \mu_{H_x}
\]
where $S$ denotes the simplex 
\[
S = \{(r_1,\dots,r_g) \in \mathbf{R}^g: -2 \leq r_1 \leq \cdots \leq r_g \leq 2\}
\]
and $H_x$ denotes the hyperplane $r_1 + \cdots + r_g = x$. 
Let us write this as an iterated integral over $r_1,\dots,r_{g-1}$, substituting $r_g = x - r_1 - \cdots - r_{g-1}$;
the endpoints of integration of $r_j$ are then
\[
\max\{r_{j-1}, x - 2(g-j) - \sum_{k=1}^{j-1} r_k \}, \qquad \min\{
2, 
\left( x - \sum_{k=1}^{j-1} r_k \right)/(g-j+1) \}
\]
(writing $r_0 = -2$).
In particular, the distribution function is continuous, even, and piecewise polynomial: on each interval $[-2g+4(i-1), -2g+4i]$ for $i=1,\dots,g$, it is a polynomial of degree $(g-1)(g+2)/2$ with rational coefficients.
For the extreme values $i=1$ and $i=g$, this polynomial is a scalar multiple of $(2g-|x|)^{(g-1)(g+2)/2}$.

Using Mathematica, we computed the distribution functions $f_g(x)$ for $g \leq 4$:
\begin{align*}
g=1:& \begin{cases} \frac{1}{4} & |x| \leq 2 \\
0 & |x| > 2 
\end{cases}
 \\
g=2:& 
\begin{cases}
\frac{3}{2^7} (4 - |x|)^2 & |x| \leq 4 \\
0 & |x| > 4
\end{cases} \\
g=3:& \begin{cases}
\frac{3}{2^{13}} (15|x|^4 - 200|x|^2 + 816)
 & |x| \leq 2 \\
\frac{3}{2^{15}} (6 - |x|)^5 & 2 < |x| \leq 6 \\
0 & |x| > 6
\end{cases}
 \\
g = 4:& \begin{cases}
 \frac{5(-|x|^9-72|x|^8-2304|x|^7+64512 |x|^6-516096 |x|^5+1548288 |x|^4-7077888
    |x|^2+24117248)}{3 \cdot 2^{27}}
     & |x| \leq 4 \\
\frac{5}{3 \cdot 2^{27}} (8 - |x|)^9 & 4 < |x| \leq 8 \\
0 & |x| > 8.
\end{cases}
\end{align*}

See Figure \ref{F:notto-tate} for plots of the distribution for $g \le 4$, and 
Figures \ref{F:notto2}, Figure \ref{F:notto3}, and Figure \ref{F:notto4}
for plots of this prediction against our data.
A table of moments for $g=3,4,5,6$ is given in Table~\ref{T:moments}. 

\begin{figure}[h]
	\begin{center}
\includegraphics[scale=0.5]{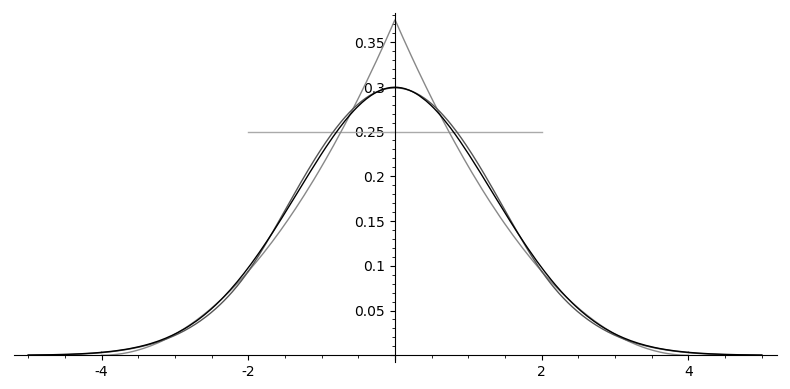}
\end{center}
\caption{Isogeny Sato-Tate distribution for $g \le 4$.\label{F:notto-tate} }
\end{figure}

It is natural to ask about the limit as $g \to \infty$. While we do not have a general formula for $f_g(x)$, we can give some evidence that the limiting distribution $f_\infty$ is a Gaussian with variance 2,
which is also the limit of the Sato-Tate distributions \cite[Theorem~6]{Diaconis1994}.\footnote{It may be feasible to prove that the limiting distribution is Gaussian by bounding the difference in moments directly.}
Let $m_\infty^{(r)} = \int_{-\infty}^\infty x^rf_\infty(x) dx$ and let $\widetilde{m}_g^{(r)}$ be a numerical approximation to $\int_{-\infty}^{\infty} x^rf_g(x) dx$.  If $f_\infty$ were Gaussian, we would have
\begin{equation} \label{eq:gaussian_relation}
m_{\infty}^{(r)} = (m_{\infty}^{(2)})^{r/2} (r-1)!!
\end{equation}
for even $r$.  As Table \ref{T:GaussComp} shows, this identity already holds approximately for $g=6$.
\begin{table}[h]
\begin{center}
\begin{tabular}{l|llllll}
$r$ & $4$ & $6$ & $8$ & $10$ & $12$ & $14$ \\ \hline \hline
$\widetilde{m}_6^{(r)}$ & $10.2041$ & $93.7929$ & $1203.9623$ & $19814.7906$ & $397315.2698$ & $9.3803 \times 10^{6}$ \\
$\left(\widetilde{m}_6^{(2)}\right)^{r/2} (r-1)!!$ & $10.2243$ & $94.3750$ & $1219.5797$ & $20263.1948$ & $411486.7236$ & $9.8754 \times 10^6$
\end{tabular}
\end{center}
\caption{Predicted moments from equation \eqref{eq:gaussian_relation}. \label{T:GaussComp} }
\end{table}

The authors don't have enough data to determine what happens when we restrict to isogeny classes of abelian varieties with a fixed Newton polygon, endomorphism algebra, delta rank, or Galois group.
We expect similar interesting distributions in these cases as $q\to \infty$ as we normalize the dimensions appropriately.
We also remark that a determination of the conjectural distribution of point counts also remains open (for say, $P(T)$ of a fixed degree $2g$ and fixed Galois group $G$).

\begin{table}[h]
\begin{center}
\begin{tabular}{l|lllllll}
$g$ &2nd & 4th & 6th& 8th & 10th & 12th & 14th \\ \hline \hline 
3	& $1.7142$ & $8.6857$ & $71.7575$ & $796.1318$ & $10750.4655$ & $166954.5839$ & $2.8786 \times 10^{6}$ \\
	4& $1.7777$ & $9.4199$ & $82.5201$ & $1001.4566$ & $15384.2906$ & $282674.8553$ & $5.9748 \times 10^{6}$ \\
	5& $1.8181$ & $9.8834$ & $89.1908$ & $1121.6573$ & $18035.9973$ & $351973.2932$ & $8.0435 \times 10^{6}$ \\
	6& $1.8461$ & $10.2041$ & $93.7929$ & $1203.9623$ & $19814.7906$ & $397315.2698$ & $9.3803 \times 10^{6}$ \\
\end{tabular}
\end{center}
\caption{Even moments for the isogeny Sato-Tate distributions for $g=3,4,5,6$. \label{T:moments} }
\end{table}

\subsection{Maximal and Minimal Point Counts} \label{subsec:maximal-minimal}
The question of the maximum number of points on a curve of given genus over a given finite field has been studied quite extensively, due to its connection with error-correcting codes via the Goppa construction. The web site \url{https://manypoints.org} tabulates most known results on this question.
In this section we investigate minimal and maximal point counts of abelian varieties. 

We will say an isogeny class $[A]$ is \emph{maximal} (resp.\ \emph{minimal}) for a fixed $g$ and $q$ when $\#A(\FF_q)$ is the maximum (resp.\ minimum) of $$\lbrace \#B(\FF_q) : \mbox{ $B$ an abelian variety of dim $g$ over $\FF_q$} \rbrace$$

While studying maximal (resp.\ minimal) abelian varieties is formally similar to studying maximal (resp.\ minimal) curves, it is not directly related: an isogeny class being maximal (resp.\ minimal) has nothing to do with it containing the Jacobian of some maximal (resp.\ minimal) curve. 
See Example \ref{E:maxmin} for explicit examples. 

In a similar vein one might also naively expect that, since Artin-Schreier curves are all maximal and known to be supersingular (see Example~\ref{E:AS}), an abelian variety $A$ being maximal (resp.\ minimal) might imply that $A$ is supersingular. 
Indeed, in some references (such as \cite{Karemaker2017}) the terms \emph{maximal} and \emph{minimal} are used to refer exclusively to supersingular abelian varieties.
However, we will see below that a maximal (resp.\ minimal) abelian variety in our sense need not be supersingular.

When looking at the data our first observation is that for a fixed $g,q$ we almost always found that there were unique isogeny classes $A_{\max}$ and $A_{\min}$ with $L(A_{\max},T) = L(A_{\min},-T)$. 
This equality implies that $A_{\min}$ and $A_{\max}$ are quadratic isogeny twists. 
\begin{example}
When $g=3$ and $q=3$ we have $P(A_{\min},T) = T^6 -9T^5+ 36T^4 -81T^3, 108T^2 -81T +27$ which is \avlink{3.3.aj\_bk\_add} and $P(A_{\max},T) = T^6 +9T^5+ 36T^4 +81T^3+108T^2 +81T +27$, which is \avlink{3.3.j\_bk\_dd}.
\end{example}

We now explain our findings. In Table~\ref{table:genus 3 maximal minimal}, we report the unique minimal and maximal isogeny classes for $g=3$ and $3 \leq q \leq 25$. These are all isogenous to cubes of elliptic curves, 
and the minimal and maximal examples for a given $q$ are isogeny twists of each other;
but some are ordinary and some are supersingular.
We omit $q=2$ because it is a bit anomalous: there are 7 minimal isogeny classes (and a unique maximal one). See Lemma~\ref{L:strict bound} and Lemma~\ref{L:min max ordinary supersingular}
for an explanation of these observations.

\begin{table}
	\begin{center}
\begin{tabular}{lllll}
&$\#A(\FF_q)$ & $P_A(T)$ & Newton Polygon & Jac? \\ \hline
$q=3$: & $343$ & ${\left(T^{2} + 3 \, T + 3\right)}^{3}$ & ss & No \\
& $1$ & ${\left(T^{2} - 3 \, T + 3\right)}^{3}$ & ss & No \\
$q=4$: & $729$ & ${\left(T + 2\right)}^{6}$ & ss & No \\
& $1$ & ${\left(T - 2\right)}^{6}$ & ss & No \\
$q=5$: & $1000$ & ${\left(T^{2} + 4 \, T + 5\right)}^{3}$ & ord & No \\
& $8$ & ${\left(T^{2} - 4 \, T + 5\right)}^{3}$ & ord & No \\
$q=7$: & $2197$ & ${\left(T^{2} + 5 \, T + 7\right)}^{3}$ & ord & No \\
& $27$ & ${\left(T^{2} - 5 \, T + 7\right)}^{3}$ & ord & No \\
$q=8$: & $2744$ & ${\left(T^{2} + 5 \, T + 8\right)}^{3}$ & ord & ??? \\
& $64$ & ${\left(T^{2} - 5 \, T + 8\right)}^{3}$ & ord & No \\
$q=9$: & $4096$ & ${\left(T + 3\right)}^{6}$ & ss & ??? \\
& $64$ & ${\left(T - 3\right)}^{6}$ & ss & No \\
$q=11$: & $5832$ & ${\left(T^{2} + 6 \, T + 11\right)}^{3}$ & ord & No \\
& $216$ & ${\left(T^{2} - 6 \, T + 11\right)}^{3}$ & ord & No \\
\end{tabular}
\end{center}
\caption{Maximal and minimal abelian varieties ($g=3$).}
\label{table:genus 3 maximal minimal}
\end{table}

We now investigate these observations.
Recall that the (sharpened) Weil bounds on an abelian variety $A$ of dimension $g$ over $\FF_q$ take the form
\begin{equation}\label{E:weil-bounds}
\lceil (\sqrt{q}-1)^2 \rceil \leq \#A(\FF_q)^{1/g} \leq \lfloor (\sqrt{q}+1)^2 \rfloor.
\end{equation}
These are precisely the maximal and minimal values appearing in Table~\ref{table:genus 3 maximal minimal}. 
We are thus led to ask whether equality in the upper (resp.\ lower) bounds happens only for a power of a maximal (resp.\ minimal) elliptic curve, or equivalently whether the inequalities become strict if we restrict to simple abelian varieties of dimension greater than 1.
This is in fact claimed in both \cite[Th\'eor\`eme~1.1]{Aubry2012} and \cite[Corollary~2.2, Corollary~2.14]{Aubry2013}, but we have already seen by an example that this is false for $q=2$; moreover, Theorem~\ref{T:Madan-Pal} implies that there are infinitely many $g$ for which there exists a simple abelian variety $A$ of dimension $g$ over $\FF_2$ with $\#A(\FF_2) = 1$.

On the other hand, by working more closely through the literature, we can recover an argument that the inequalities become strict. 
\begin{lemma} \label{L:strict bound}
	For $q=5,7$, assume that $\#A(\FF_q)^{1/g}>N_q$ for every abelian variety $A$ of dimension $g$ defined over $\FF_q$, where $N_5=2.708$ and $N_7 =3.970$ (reported to us by Kadets).
	Then for all $q > 2$, the inequalities \eqref{E:weil-bounds} become strict for simple abelian varieties of dimension greater than $1$.
\end{lemma}
\begin{proof}
First, by inspection of the proof of \cite[Proposition~2.17]{Aubry2013}, one deduces that
the lower bound is strict for $q \geq 8$; thus only the cases $q=3,4,5,7$ are at issue.
For $q=3$ and $q=4$, Theorem~\ref{T:Kadets small-q bound} implies that the lower bound is strict
(for $q=4$ we may also apply Theorem~\ref{T:Kadets general bound}).
For $q=5$ and $q=7$, one can obtain similar lower bounds by emulating the calculation used to prove
Theorem~\ref{T:Kadets small-q bound}. 
Kadets reports that a nonrigorous version of the calculation
gives the lower bounds $2.708$ for $q=5$ and $3.970$ for $q=7$, but as of this writing a rigorous calculation remains to be made.
\end{proof}

We raise our observations to the status of theorems below.
\begin{lemma}\label{L:min max ordinary supersingular}
Under the hypothesis of Lemma~\ref{L:strict bound},
for every $g>1$ and every $q>2$ there exists unique maximal and minimal isogeny classes of dimension $g$. 
	These classes are quadratic isogeny twists of each other and are a power of the unique maximal (resp.\ minimal) isogeny class of elliptic curves.
	Finally, the class is supersingular or ordinary according to whether $p$ divides $\lfloor 2 \sqrt{q} \rfloor$.
\end{lemma}
\begin{proof}
	By Theorem~\ref{T:Kadets general bound}, one always gets a maximal (resp.\ minimal) abelian variety of a given dimension over $\FF_q$ by taking a power of the maximal (resp.\ minimal) elliptic curve over $\FF_q$.
	By the discussion following equation \eqref{E:weil-bounds}, this is unique except for the minimal case over $\FF_2$ in some genera--- the isogeny class of an elliptic curve is determined by a single point count.
	In particular, since the minimal and maximal elliptic curve over $\FF_q$ are isogeny twists of each other,
	the same is true of their powers.
\end{proof}

\begin{remark}
As asserted in Lemma~\ref{L:min max ordinary supersingular},
		whether or not the maximal and minimal elliptic curves over $\FF_q$, 
		and hence the resulting abelian varieties, are supersingular or ordinary
	depends on whether or not $\lfloor 2 \sqrt{q} \rfloor$ is divisible by $p$. 
	When $q = p$, this divisibility holds only for $p=2,3$ (as otherwise $0 < 2 \sqrt{p} < p$).
	When $q$ is a square, so $q = p^{2e}$, we have $\lfloor 2 \sqrt{q} \rfloor = 2 \sqrt{p^{2a}}=2p^a$ which is obviously divisible by $p$.
	In other cases, $q = p^{2e+1}$ for some positive integer $e$ and we are asking whether the base-$p$ expansion of $2\sqrt{p}$ has a zero in the $e$-th position after the radix point; this can occur but is rather sporadic.\footnote{The \emph{radix point} is the analogue of the decimal point for a general base expansion. 
	}
\end{remark}

Following the discussion in Subsection~\ref{ssec:point_count_bounds}, we now restrict attention to simple abelian varieties. 
We say that an isogeny class $[A]$ of dimension $g$ over $\FF_q$
is \emph{simple-maximal} (resp.\ \emph{simple-minimal}) if
$\#A(\FF_q)$ is maximal (resp.\ minimal) among simple abelian varieties of dimension $g$ over $\FF_q$.
In Table~\ref{T:simple-maxmin}, we report the unique simple-maximal and simple-minimal isogeny classes
for $g=3$ and $5 \leq q \leq 25$. For each of $q=2,3,4$, there are 2 simple-minimal isogeny classes (and a unique simple-maximal one).

We make a few curious observations about this data which we are unable to rigorously explain. For one, all of the simple-maximal and simple-minimal examples are ordinary. For another, the simple-maximal and simple-minimal varieties are most often isogeny twists of each other, but not always (see $q=5, 13, 17, 19$);
in any case, they have opposite sign patterns (the coefficients of $P_A(T)$ are positive for $A$ simple-maximal and alternate in sign for $A$ simple-minimal). 
Finally, note that the simple-maximal variety for $q=3$ is a Jacobian.

\begin{table}[h]
\begin{tabular}{lllll}
    &$\#A(\FF_q)$ & $P_A(T)$ & NP & Jac? \\ \hline
	$q=5$: & $631$ & $T^{6} + 8 \, T^{5} + 34 \, T^{4} + 93 \, T^{3} + 170 \, T^{2} + 200 \, T + 125$ & ord & Yes \\
	& $25$ & $T^{6} - 8 \, T^{5} + 32 \, T^{4} - 85 \, T^{3} + 160 \, T^{2} - 200 \, T + 125$ & ord & No \\
	$q=7$: & $1561$ & $T^{6} + 11 \, T^{5} + 59 \, T^{4} + 195 \, T^{3} + 413 \, T^{2} + 539 \, T + 343$ & ord & ??? \\
	& $71$ & $T^{6} - 11 \, T^{5} + 59 \, T^{4} - 195 \, T^{3} + 413 \, T^{2} - 539 \, T + 343$ & ord & No \\
	$q=8$: & $2157$ & $T^{6} + 12 \, T^{5} + 69 \, T^{4} + 243 \, T^{3} + 552 \, T^{2} + 768 \, T + 512$ & ord & ??? \\
	& $111$ & $T^{6} - 12 \, T^{5} + 69 \, T^{4} - 243 \, T^{3} + 552 \, T^{2} - 768 \, T + 512$ & ord & No \\
	$q=9$: & $2911$ & $T^{6} + 13 \, T^{5} + 81 \, T^{4} + 305 \, T^{3} + 729 \, T^{2} + 1053 \, T + 729$ & ord & ??? \\
	& $169$ & $T^{6} - 13 \, T^{5} + 81 \, T^{4} - 305 \, T^{3} + 729 \, T^{2} - 1053 \, T + 729$ & ord & No \\
	$q=11$: & $4861$ & $T^{6} + 15 \, T^{5} + 105 \, T^{4} + 439 \, T^{3} + 1155 \, T^{2} + 1815 \, T + 1331$ & ord & ??? \\
	& $323$ & $T^{6} - 15 \, T^{5} + 105 \, T^{4} - 439 \, T^{3} + 1155 \, T^{2} - 1815 \, T + 1331$ & ord & No \\
	$q=13$: & $7181$ & $T^{6} + 16 \, T^{5} + 122 \, T^{4} + 555 \, T^{3} + 1586 \, T^{2} + 2704 \, T + 2197$ & ord & ??? \\
	& $615$ & $T^{6} - 16 \, T^{5} + 120 \, T^{4} - 543 \, T^{3} + 1560 \, T^{2} - 2704 \, T + 2197$ & ord & No \\
	$q=16$: & $12649$ & $T^{6} + 19 \, T^{5} + 166 \, T^{4} + 847 \, T^{3} + 2656 \, T^{2} + 4864 \, T + 4096$ & ord & ??? \\
	& $1189$ & $T^{6} - 19 \, T^{5} + 166 \, T^{4} - 847 \, T^{3} + 2656 \, T^{2} - 4864 \, T + 4096$ & ord & No \\
	$q=17$: & $14351$ & $T^{6} + 19 \, T^{5} + 169 \, T^{4} + 885 \, T^{3} + 2873 \, T^{2} + 5491 \, T + 4913$ & ord & ??? \\
	& $1539$ & $T^{6} - 19 \, T^{5} + 167 \, T^{4} - 871 \, T^{3} + 2839 \, T^{2} - 5491 \, T + 4913$ & ord & No \\
	$q=19$: & $19601$ & $T^{6} + 21 \, T^{5} + 201 \, T^{4} + 1119 \, T^{3} + 3819 \, T^{2} + 7581 \, T + 6859$ & ord & ??? \\
	& $2113$ & $T^{6} - 21 \, T^{5} + 197 \, T^{4} - 1085 \, T^{3} + 3743 \, T^{2} - 7581 \, T + 6859$ & ord & No \\
	$q=23$: & $32671$ & $T^{6} + 24 \, T^{5} + 258 \, T^{4} + 1591 \, T^{3} + 5934 \, T^{2} + 12696 \, T + 12167$ & ord & ??? \\
	& $4049$ & $T^{6} - 24 \, T^{5} + 258 \, T^{4} - 1591 \, T^{3} + 5934 \, T^{2} - 12696 \, T + 12167$ & ord & ??? \\
	$q=25$: & $40391$ & $T^{6} + 25 \, T^{5} + 281 \, T^{4} + 1809 \, T^{3} + 7025 \, T^{2} + 15625 \, T + 15625$ & ord & ??? \\
	& $5473$ & $T^{6} - 25 \, T^{5} + 281 \, T^{4} - 1809 \, T^{3} + 7025 \, T^{2} - 15625 \, T + 15625$ & ord & ??? \\
\end{tabular}
\caption{A table of the unique simple-maximal and simple-minimal Weil polynomials for $g=3$.}
\label{T:simple-maxmin}
\end{table}

\begin{table}[ht]
	\begin{tabular}{llp{4in}ll}
		&$\#A(\FF_q)$ & $P_A(T)$ & NP & Jac? \\ \hline
		$q=3$: & $979$ & $T^{8} + 7 \, T^{7} + 27 \, T^{6} + 72 \, T^{5} + 143 \, T^{4} + 216 \, T^{3} + 243 \, T^{2} + 189 \, T + 81$ & ord & ??? \\
		& $5$ & $T^{8} - 6 \, T^{7} + 13 \, T^{6} - 10 \, T^{5} + T^{4} - 30 \, T^{3} + 117 \, T^{2} - 162 \, T + 81$ & ord & No \\
		$q=4$: & $2521$ & $T^{8} + 9 \, T^{7} + 42 \, T^{6} + 132 \, T^{5} + 305 \, T^{4} + 528 \, T^{3} + 672 \, T^{2} + 576 \, T + 256$ & ord & ??? \\
		& $29$ & $T^{8} - 9 \, T^{7} + 41 \, T^{6} - 125 \, T^{5} + 285 \, T^{4} - 500 \, T^{3} + 656 \, T^{2} - 576 \, T + 256$ & ord & No \\
		$q=5$: & $5599$ & $T^{8} + 11 \, T^{7} + 62 \, T^{6} + 229 \, T^{5} + 601 \, T^{4} + 1145 \, T^{3} + 1550 \, T^{2} + 1375 \, T + 625$ & ord & ??? \\
		& $61$ & $T^{8} - 10 \, T^{7} + 45 \, T^{6} - 130 \, T^{5} + 305 \, T^{4} - 650 \, T^{3} + 1125 \, T^{2} - 1250 \, T + 625$ & ord & No \\
	\end{tabular}
	\caption{A table of the unique simple-maximal and simple-minimal Weil polynomials for $g=4$. }
\end{table}

Finally, in Table~\ref{table:compare point counts to Kadets bounds}, we compare the extreme values of $\#A(\FF_q)^{1/g}$
to the bounds given in Theorem~\ref{T:Kadets small-q bound}.

\begin{table}[ht]
\begin{center}
\begin{tabular}{cccccc}
$q$ & $g$ & lower bound & minimum & maximum & upper bound \\
\hline
2 & 4 & 1 & 1 & 3.940 & 4.035\\
2 & 5 & 1 & 1.149 & 3.717 & 4.035\\
2 & 6 & 1 & 1 & 3.697 & 4.035 \\
3 & 4 & 1.359 & 1.495 & 5.594 & 5.634 \\
3 & 5 & 1.359 & 1.670 & 5.423 & 5.634 \\
4 & 4 & 2.275 & 2.321 & 7.086 & 7.382 \\
5 & 4 & 2.708 & 2.795 & 8.615 & 8.938\\
\end{tabular}
\end{center}
\caption{Bounds and extreme values for $\#A(\FF_q)^{1/g}$ for a simple abelian variety $A$ of dimension $g$ over $\FF_q$.
The bounds for $q=2,3,4$ are taken from \cite{Kadets2019}; the bounds for $q=5$ were computed numerically
(but not rigorously) by Kadets using the same method (see Lemma~\ref{L:strict bound}).} 
\label{table:compare point counts to Kadets bounds}
\end{table}
 
\section{An Isogeny Class Scavenger Hunt}
\label{S:examples}

In this section, we describe a number of examples related to questions or results in the literature.

Several of these examples involve Jacobians of curves, whereas the LMFDB does not currently contain
complete information about Jacobians of curves of genus at least 4 (\S \ref{ssec:Jacobians}).
To generate these examples, we used Sage \cite{sage} to exhaust over hyperelliptic curves, computing zeta functions until we found one of the desired form. (This can also be done in Magma \cite{magma}.)

\subsection{Some Basic Examples}

\begin{example}[supersingular elliptic curves]\label{exa:supersingular} One of the first interesting examples of abelian varieties are supersingular elliptic curves defined over finite fields.
If $E$ is an elliptic curve defined over a finite field, we have 
 $$\zeta(E/\FF_q,T) = \frac{1-aT + qT^2}{(1-T)(1-qT)},$$
with $a \in \ZZ$ and $\vert a \vert < 2 \sqrt{q}$ by the Hasse-Weil bound. In addition, if $E$ is supersingular then $a \equiv 0 \mod p$, for $p$ the characteristic of $\FF_q$.
Therefore if $p$ is prime and $p \geq 5$, then there is a unique isogeny class of supersingular elliptic curves over $\FF_p$, namely the one with $a=0$.
This means that
 $$P_E(T)= T^2 + q =(T - i \sqrt{q})(T+i \sqrt{q});$$ 
 a similar phenomenon holds in more generality. Indeed, when $A/\FF_q$ is a simple abelian variety, $A$ is supersingular if and only if $\pi_A = \zeta \sqrt{q}$ for $\zeta$ a root of unity: namely, if $A$ is supersingular,
 then $\pi/\sqrt{q}$ is an algebraic number with norm 1 in every finite and infinite place, so by Kronecker--Weber it is a root of unity. 
 
In contrast, while an elliptic curve defined over a finite field has $p$-rank zero if and only if it is supersingular, this does not hold in higher dimension, even for simple abelian varieties. 
For example, isogeny class \avlink{3.2.ac\_c\_ac} is geometrically simple of $p$-rank 0, but has Newton polygon slopes $[1/3, 1/3, 1/3, 2/3, 2/3, 2/3]$ and therefore is not supersingular.
(This is the smallest dimension for which $p$-rank 0 does not imply supersingular.)
Conversely, a supersingular isogeny class always has $p$-rank 0.

We finally comment on the endomorphism rings of supersingular elliptic curves.
While every supersingular elliptic curve has geometric endomorphism ring a maximal order in a quaternion algebra, it is not the case that every endomorphism is defined over the field of definition of the elliptic curve.
For example, there are supersingular elliptic curves defined over $\FF_p$ for each prime $p$ (belonging to the isogeny class with $a=0$, as above), but each of them has endomorphism ring isomorphic to an order in an imaginary quadratic field and only acquires its extra endomorphisms over an extension of $\FF_p$.
For example, isogeny class \avlink{1.3.ad} has endomorphism degree $6$, by which we mean that its geometric endomorphisms are defined over $\FF_{3^6}$; isogeny class \avlink{1.2.ac} has endomorphism degree $4$; and isogeny class \avlink{1.2.a} has endomorphism degree $2$.
\end{example}

\begin{example}
We now give examples of particularly uncomplicated Weil polynomials for $g=2$ and $g=4$.
These polynomials can be useful to check conjectures when one wants to produce examples quickly by hand.
We first observe that by Lemma~\ref{L:weil-q-number-characterization} a Weil polynomial $P_A(T)=P(T)$ of an abelian variety $A$ of dimension $g$ can be written as 
\begin{equation}\label{E:weil-poly-example}
P(T)= \prod_{i=1}^g(T^2 -\beta_i T + q)
\end{equation} 
 where $\beta_i = \alpha_i + q/\alpha_i$ is totally real, and $\alpha_1,\ldots,\alpha_{2g}$ are the roots of the Weil polynomial, arranged so that $\alpha_i \alpha_{2g-i} = q$ for $i=1,\dots,g$.

When $g=2$, we can expand this to obtain:
 $$P(T) = T^4 -(\beta_1+\beta_2)T^3 + (2q+\beta_1\beta_2)T^4-(\beta_1+\beta_2)q T + q^2.$$
For $r$ a positive integer, choosing $\beta_1= \sqrt{r}$ and $\beta_2 = - \sqrt{r}$ gives a nice family of examples.
In fact, for any positive integer $a$ not divisible by $p$
such that $2q>|a|$, letting $r = 2q-a$ gives the characteristic polynomial
 $$P(T) = T^4 +aT^2 + q^2 $$
for an abelian surface.
(Note that we require $p \nmid a$ to ensure that $e_A = 1$; see Section~\ref{subsec:Honda-Tate}.)
This isogeny class has Weil $q$-numbers
$$+\sqrt{ \frac{-a + \sqrt{a^2 - 4q^2} }{2}  },+\sqrt{ \frac{-a - \sqrt{a^2 - 4q^2} }{2}  },-\sqrt{ \frac{-a + \sqrt{a^2 - 4q^2} }{2}  },-\sqrt{ \frac{-a -\sqrt{a^2 - 4q^2} }{2}  }.$$
Even more concretely, for $q=5$ any $a \in \{-9,-8, \ldots, 8, 9\}$
is admissible. In the case $a=-1$ this is \avlink{2.5.a\_ab}. 

One can generalize the above example by considering \eqref{E:weil-poly-example} where $\beta_1,\dots,\beta_g$ are totally real algebraic integers which form a union of Galois orbits over $\QQ$
 and $q$ is a sufficiently large prime power.

Another pleasing family of isogeny classes can be obtained by considering the case of $g=4$ with two isogeny factors of dimension $2$ each with Weil polynomial of the form $T^4+a_iT^2+q^2$ as above. 
In this case this isogeny class has Weil polynomial
$$ (T^4+a_1T^2+q^2)(T^4+a_2T^2+q^2) = T^8  +(a_1+a_2)T^6 + (2q^2+a_1a_2)T^4+(a_1+a_2)q^2 T + q^4,$$
which we can specialize to $a_1 = c + b\sqrt{d}$ and $a_2 = c-b\sqrt{d}$, where $a_1$ and $a_2$ are algebraic integers, and as before $2q> |a_i|$ for each $i$, to give 
$$ a_1 + a_2 = 2c, \qquad a_1a_2 = c^2-db^2,$$
and 
$$P(T) = T^8 +2cT^6 + (2q^2 +c^2-db^2)T^4 +2cq^2T+q^4.$$
In the special case $c=1,b=2,d=2,q=5$ this is \avlink{4.5.a\_c\_a\_br}.
\end{example}

\begin{example}[{\cite[Example 2.1]{Achter2011}} and {\avlink{4.3.a\_a\_a\_g}}]
Consider the hyperelliptic curve given by 
 $$ C: y^2 = x^9 + x^7 + 2x^5 + x^4 + 2x^3 + x^2 + x. $$
The genus of this curve is $4$ and 
 $$ L(C/\FF_3,T) = 81T^8 + 6T^4 + 1. $$
This example was computed by writing down the shape of the zeta function and equating terms in the truncated Taylor series.
\end{example}

\begin{example}[{\cite[Example 2.3]{Voight2005}} and {\avlink{3.2.a\_a\_f}} ] 
The isogeny class of the Jacobian of a curve $C$ over $\FF_q$ can be identified by knowing $\#C(\FF_{q^r})$ for $1\leq r \leq g$. 
In the examples where $C$ is a projective curve defined by $x^3y+x^3z+y^3z=0$ or $x^3y+y^3z+xz^3 = 0$ over $\FF_2$, $C$ has genus 3 and we can compute point counts ``by hand'' using a presentation of the field:
\begin{center}
\begin{tabular}{ c || c | c | c | c }
  $r$ & 1 & 2 & 3 & 4  \\
  \hline $\#C(\FF_{2^r})$ & 3 & 4 & 24& 17 
\end{tabular}
\end{center}
This allows us to solve explicitly for the $L$-polynomial (see \emph{op.\ cit}.)
  $$ L(T) = 8T^6 +5T^3 + 1.$$
\end{example}

\begin{example}\label{E:AS}
Artin-Schreier curves are fan favorites.
An affine Artin-Schreier curve has a model
  $$U_{f,q}: y^q-y = f_d(x)$$ 
where $f_d(x) \in \FF_q[x]$ is a polynomial of degree $d$. 
After a desingularization of the naive projective model one gets a proper model $X=X_{f,q}$.
In the case where $q=p$, the genus of this curve is $g=(d-1)(p-1)/2$, and for $d=5$ and $p=3$ the curve has genus $4$ so its Jacobian will be in our database.
How can we find it?
It turns out that the point counts of such a curve can be explicitly computed:
 $$ \#X_{f,q}(\FF_{q^n}) = 1+ \left ( \sum_{\psi \in \widehat{\FF_q}\setminus \lbrace 1 \rbrace } S(\psi,n) \right) + q^n. $$
In the above expression the $\psi$'s are additive characters and the $S$'s are character sums given by
$$S(\psi,n)  := 	\sum_{b \in \FF_{q^n}}\psi( \Tr_{q,n}(f(b)) ),\qquad \Tr_{q,n}(a) := \sum_{i=0}^{n-1} a^{q^i}.$$ 
All of the additive characters of $\FF_q$ are parametrized by $a \in \FF_p$ and take the form $\chi_a\colon \FF_q \to \CC$ where
$$ \chi_a(b) = \zeta_p^{a \Tr_{\FF_{q}/\FF_p}(b) }. $$
For details on this computation see \cite[\S2.1 and \S2.2]{Mueller2013}.

In the special case where $f(x) = x^5$, 
we computed the values of these characters using Sage, and found that $\#X(\FF_{3^n})=4,10,28,154,244$ for $n=1,2,3,4,5$. 
Searching our database gives \avlink{4.3.a\_a\_a\_s}, as the isogeny class of the Jacobian.
This isogeny class contains supersingular abelian varieties, as expected. 

In the special case $f(x) = x^5+2x+1$, 
we again compute $\#X(\FF_{3^n})=1,7,55,91,271$ for $n=1,2,3,4,5$, and find that the Jacobian belongs to \avlink{4.3.ad\_d\_j\_abb}, 
another non-simple supersingular isogeny class; this time the isogeny factors are not isogenous to each other. 

\end{example}

\begin{example}
The isogeny class \avlink{2.2.a\_ad}
provides an example of a principally polarizable isogeny class of abelian varieties of dimension 2 over $\FF_2$ which does not contain a Jacobian (Example~\ref{exa:Shankar-Tsimerman} is another example in higher dimension). 

Also, as discussed previously, there exist principally polarizable abelian varieties which have isogeny factors whose isogeny classes are not principally polarizable.
For instance, the $g=4,q=5$ isogeny class \avlink{4.5.ak\_bp\_adq\_hc}
has a $g=2$ isogeny factor 
$\avlink{2.5.ac\_ab}$
which contains no principally polarizable abelian surface.

For interested readers, we point out that our methods don't allows us to determine if the particular $g=4$ simple isogeny class \avlink{4.5.ag\_o\_au\_bj} contains a principally polarizable abelian variety. 
We are unable to do so whenever the isogeny class is not ordinary, and its associated CM field $K=\QQ(\pi)$ (which in this particular case is \nflink{8.0.268960000.3}) is unramified over its totally real subfield $K_+$ and every prime of $K_+$ dividing $\pi-q/\pi$ is not inert in $K/K_+$. 
In short, the condition in \cite[Theorem 1.1]{Howe1996} summarized in \S\ref{S:principal-polarizations} does not completely answer the question: ``Does this non-ordinary isogeny class contain a principally polarizable variety?"
\end{example}

\begin{remark}
Howe points out to us that the 2-dimensional isogeny class  \avlink{2.2.a\_ad} presented in the above example is one from a family that is in the Appendix to \cite{Maisner2002}.
The theorem in \emph{op.\ cit}.\ is that no polynomial of the form $T^4+(1-2q)T^2 + q^2$ is the Weil polynomial of a Jacobian of a curve over $\FF_q$. 
One can also see that such an isogeny class contains a principally polarized abelian variety. It is the restriction of scalars from $\FF_{q^2}$ to $\FF_q$ of an elliptic curve with Weil polynomial $T^2 + (1-2q)T+ q^2$.
\end{remark}

\begin{example}\label{E:maxmin}
	In this example we show that maximality or minimality of a curve has nothing to do with the associated Jacobian being maximal or minimal as an abelian variety in the sense of \S \ref{subsec:maximal-minimal}. 
	Indeed, the isogeny class \avlink{3.3.aj\_bk\_add} is minimal for $q=3$ and $g=3$ with $\#A(\FF_3)=1$, but it does not contain a Jacobian since its virtual curve count has $\#C(\FF_3)=-5$.
	Therefore, a minimal curve of genus $3$ over $\FF_3$ cannot have minimal Jacobian.
\end{example}

\begin{example}
	When $g=6$ and $q=2$ there exist isogeny twists \avlink{6.2.a\_ac\_a\_c\_a\_a} and \avlink{6.2.a\_c\_a\_c\_a\_a} which have different number fields (with the same
	discriminant).  
	The two number fields have a different number of places
	above 2 (3 in one case and 4 in the other), so the Brauer invariants of
	their endomorphism algebras are different: $(0,0,0)$ vs $(0,0,0,0)$.
\end{example}

\begin{example}[communicated to us by Howe]\label{E:isogenoustopp}
Consider isogeny class \avlink{2.2.b\_b}, which has characteristic polynomial $P(T)=T^4+T^3+T^2+2T+4$. 
The ring $\ZZ[\pi,q/\pi]$ is a maximal order in a field with class number one, and therefore this isogeny class contains a single isomorphism class of abelian varieties.
Furthermore, this abelian variety must have a principal polarization (by \cite{Howe1995} or by the more general Theorem 1 of \cite{Howe1996}). This shows that while over an algebraically closed field of characteristic 0 every abelian variety is isogenous to an abelian variety that does not admit a principal polarization, this is not true in general (cf. \S \ref{S:principal-polarizations}).
\end{example}

\subsection{Supersingular Curves}

It is not clear whether every eligible Newton polygon can occur for curves of a given characteristic (see Remark \ref{rem:NSdim}), and the following is a related long-standing open problem.
\begin{question}
For every prime $p$ and every positive integer $g$, does there exist a genus $g$ curve over $\overline{\FF}_p$ whose Jacobian is supersingular?
\end{question}

This is known for $g \leq 4$ \cite{Kudo2019}.
It is also known for all $g$ when $p=2$ \cite{VanderGeer1995};
however, these curves cannot always be taken to be hyperelliptic \cite{Scholten2002}.
Some higher-genus cases are treated in \cites{Li2018, Li2019}; however, the following example is not covered by those papers (nor by \cite{VanderGeer1995}, which handles some genera for $p>2$).

\begin{example} \label{exa:hyperelliptic supersingular genus 5}
Let $C$ be the hyperelliptic curve over $\FF_3$ given by
\[
y^2 = x^{11} + 2x^9 + x^5 + x^3 + x.
\]
Then
\[
L(C/\FF_3, T) = 1 + 3T^2 + 81T^8 + 243T^{10},
\]
so the Jacobian $A$ of $C$ is supersingular and belongs to isogeny class \avlink{5.3.a\_d\_a\_a\_a}.
\end{example}

\begin{remark}
One can also consider curves whose Jacobians are \emph{superspecial}, meaning that they are isomorphic
(not just isogenous) to a product of supersingular elliptic curves. 
Ekedahl \cite{Ekedahl1987} showed there do not exist superspecial genus $4$ curves in characteristics $2$ and $3$, and asked whether conversely they exist in all characteristics at least 5; this question was answered negatively by Kudo--Harashita \cite{Kudo2017}, who showed that none exist in characteristic $7$. 
\end{remark}

\begin{remark} \label{rem:NSdim}
The moduli space $\mathcal{M}_g$ of curves of genus $g$ has dimension $3g-3$, which is much smaller than the length of a maximal chain of inclusions of Newton strata in $\A_g$, which is asymptotic to $g^2/4$ (see Corollary \ref{cor:NPcat}).  However, unlike $\A_g$, which has irreducible strata except for the supersingular case, Newton strata in $\mathcal{M}_g$ can be much more reducible.  While it is conceivable that there are enough strata in $\mathcal{M}_g$ to account for all possible Newton polygons, it does not seem likely.\end{remark}

\subsection{Ordinarity and Angle Ranks}

The following is a conjecture of Ahmadi--Shparlinski.
\begin{conjecture}[{\cite[\S 5]{Ahmadi2010}}] \label{conj:Ahmadi-Shparlinski}
Every ordinary geometrically simple Jacobian has maximal angle rank. 
\end{conjecture}

This is verified in the database in dimensions 2 and 3.
This is a theorem in dimension 2, even without the ordinary condition: \cite[Theorem~2]{Ahmadi2010}.
It is also a theorem in dimension 3, but this time it requires the ordinary condition:
\cite[Theorem~1.1]{Zarhin2015}.

We verified the conjecture in dimension 4 over $\FF_2$ as follows.
According to the LMFDB, there are 52 isogeny classes of ordinary, geometrically simple abelian varieties with angle rank at most 3 (in fact they are all equal to 3). 
Since the LMFDB does not yet contain full data about whether an isogeny class in dimension greater than 3 contains a Jacobian, we used the fact that every nonhyperelliptic genus 4 curve is the intersection of a quadric and a cubic in $\PP^3$
to compute the zeta functions of all genus 4 curves over $\FF_2$. 
We found 620 distinct zeta functions, none of which occur among the previous list of 52. 
(This result has been independently confirmed by Xarles \cite{Xarles2020}.)

As an aside, note that \cite[Theorem~1.1]{Zarhin2015} implies that for each of the 52 isogeny classes in the previous paragraph, the endomorphism algebra must contain an imaginary quadratic field, which we have confirmed.

By contrast, the conjecture fails in dimension 4 over $\FF_3$ and $\FF_5$, as shown by the curves given in Example \ref{exa:Ahmadi-Shparlinski}. 
We also note in passing that Conjecture~\ref{conj:Ahmadi-Shparlinski} is incompatible with Conjecture~\ref{conj:Shankar-Tsimerman} below.

\begin{example}
\label{exa:Ahmadi-Shparlinski}
Let $C$ be the hyperelliptic curve over $\FF_3$ given by
\[
y^2 = x^9 + x^8 + x^7 + 2x^5 + x.
\]
Then
\[
L(C/\FF_3, T) = 
1 - T + 2T^2 - 4T^3 - 2T^4 - 12T^5 + 18T^6 - 27T^7 + 81T^8,
\]
so the Jacobian $A$ of $C$ belongs to isogeny class \avlink{4.3.ab\_c\_ae\_ac}.
We see that $A$ is ordinary, geometrically simple, and has angle rank 3. 
It thus constitutes a counterexample to the Ahmadi--Shparlinski conjecture (Conjecture~\ref{conj:Ahmadi-Shparlinski}).
Consistently with Zarhin's theorem, the endomorphism algebra contains the field $\QQ(\sqrt{-7})$.

Similarly, let $C$ be the hyperelliptic curve over $\FF_5$ given by 
\[
y^2 = x^9 + x^6 + 2x^5 + x.
\]
Then
\[
L(C/\FF_5,T) = 
1-T+2T^2-4T^3+16T^4-20T^5+50T^6-125T^7+625T^8,
\]
so the Jacobian $A$ of $C$ belongs to isogeny class \avlink{4.5.ab\_c\_ae\_q}.
Again, $A$ is ordinary, geometrically simple, and has angle rank 3.
The endomorphism algebra contains the field $\QQ(\sqrt{-15})$.
\end{example}

\subsection{Function Fields of Class Number One}

\begin{example}[{\cite{Stirpe2014}} and {\avlink{4.2.ad\_c\_a\_b}}] \label{exa:class number one}
The following is the LMFDB annotation for {\avlink{4.2.ad\_c\_a\_b}}.
In \cite{Stirpe2014}, Stirpe exhibited an example of a genus 4 curve $C/\FF_2$ for which
\[
L(C/\FF_2, T) = 1 - 3 T + 2 T^{2} + T^{4} + 8 T^{6} - 24 T^{7} + 16 T^{8}.
\]
The Jacobian $A$ of this curve belongs to isogeny class \avlink{4.2.ad\_c\_a\_b}.

This example is notable because $C$ has the largest possible genus among curves over finite fields with trivial class group, and because it refuted a published result from almost 40 years earlier. In \cite{Leitzel1975}, it was shown (correctly) that there are seven such curves of genus at most 3 
and at most one of genus 4; it was also claimed (incorrectly) that the genus 4 case could be ruled out.
Correct proofs of the classification can be found in \cite{Mercuri2015} and \cite{Shen2015}.
\end{example}

\subsection{Hypersymmetric Abelian Varieties}
Let $A$ be an abelian variety over a field $k \supset \FF_p$. 
Following Chai--Oort \cite[Definition~2.1]{Chai2006},
we say that $A$ is \emph{hypersymmetric} if  
  $$ \End(A_{\overline{k}})_{\ZZ_p} \cong \End(A_{\overline{k}}[p^{\infty}]). $$ 
These are meant to provide a positive-characteristic analogue of CM points in the moduli space of abelian varieties.

According to \cite[Theorem~3.3]{Chai2006}, for a simple abelian variety, one can read off whether it is hypersymmetric explicitly from the Frobenius polynomial.

Here is an explicit example.

\begin{example} \label{exa:hypersymmetric}
We exhibit a simple hypersymmetric abelian threefold over $\FF_8$ by verifying that
\avlink{3.8.ag\_bk\_aea} satisfies the conditions of \cite[Conclusion~3.6]{Chai2006}.
The Newton polygon has slopes $1/3$ and $2/3$ each with multiplicity 3, so it is \emph{balanced} in the sense of \cite[Definition~3.4]{Chai2006}. The prime 2 splits completely in $\QQ(\pi) = \QQ(\sqrt{-7})$, and the Brauer invariants of the endomorphism algebra at the places above 2 are again $1/3$ and $2/3$.

The same conditions are satisfied by the quadratic isogeny twist \avlink{3.8.g\_bk\_ea}. We have checked that other than elliptic curves (which are all hypersymmetric),
these are the only examples of hypersymmetric abelian varieties currently found in the LMFDB.
\end{example}

\subsection{Isomorphic Endomorphism Algebras and Different \texorpdfstring{$p$}{p}-ranks}

\begin{example}
This example is a modified version of \cite[Example~4.2]{Gonzalez1998}. (That example starts over $\FF_3$ rather than $\FF_2$, but we do not currently have abelian threefolds over $\FF_{27}$ in the LMFDB.)

Let $A$ be an abelian threefold over $\FF_2$ in the isogeny class \avlink{3.2.ad\_c\_b}. Then $A$ is simple, ordinary, and of $p$-rank 3, and its endomorphism algebra is $\QQ(\zeta_7)$. Although $A$ is not geometrically simple, its base change
to $\FF_8$, which belongs to the isogeny class \avlink{3.8.ag\_bd\_adf}, is again simple.

Let $B$ be an abelian threefold over $\FF_8$ in the isogeny class \avlink{3.8.ag\_i\_i}; the Weil number for $B$ is twice that for $A$. Here $B$ is geometrically simple of $p$-rank 0, and its endomorphism algebra is again $\QQ(\zeta_7)$.
\end{example}

\begin{example}
There are also examples where both $A$ and $B$ are geometrically simple and have the same endomorphism algebra but different $p$-ranks.  For example, abelian varieties in the isogeny class \avlink{3.2.a\_c\_c} have $p$-rank $0$ while abelian varieties in the isogeny class \avlink{3.2.d\_f\_h} have $p$-rank $3$; both have endomorphism algebra isomorphic to the number field \nflink{6.0.679024.1}.
\end{example}

\subsection{Abelian Fourfolds as Jacobians}

For $g \geq 4$, a generic principally polarized abelian variety is not isomorphic to a Jacobian for dimension reasons. However, when working up to isogeny, it becomes much less clear what to expect.
Using ideas from the theory of unlikely intersections, Shankar--Tsimerman \cite{Shankar2018} have made numerous observations about this question, including the following conjecture.

\begin{conjecture}[{\cite[Conjecture~2.5]{Shankar2018}}] \label{conj:Shankar-Tsimerman}
Every $4$-dimensional abelian variety over $\overline{\FF}_p$ is isogenous to the Jacobian of some (possibly reducible) stable curve.
\end{conjecture}

Notably, however, it is not predicted that a 4-dimensional abelian variety over $\FF_q$ is isogenous to the Jacobian of a curve over $\FF_q$. This can fail to occur, as in the following example.

\begin{example} \label{exa:Shankar-Tsimerman}
Consider the isogeny class \avlink{4.2.c\_c\_ac\_af} over $\FF_2$.
This class contains a principally polarized abelian fourfold $A$
which is not isogenous to a Jacobian: if it were, the corresponding curve
would have a negative number of $\FF_8$-points (see \S\ref{subsec:curve counts}).
However, the base change of $A$ to $\FF_4 = \FF_2[\alpha]$, which belongs to the isogeny class
\avlink{4.4.a\_c\_i\_j}, is isogenous to the Jacobian of the hyperelliptic curve
\[
y^2 + (\alpha x^4 + x^3 + x^2 + x + \alpha)y = x^9 + x^8.
\]
\end{example}

We can also stress-test the conjecture by considering the following example.
\begin{example}
Consider the isogeny class \avlink{4.2.a\_a\_b\_af} over $\FF_2$.
This isogeny class is geometrically simple and ordinary.
It cannot contain a Jacobian over any subfield of $\FF_{16}$: if it did, the corresponding curve would have
a negative number of $\FF_{16}$-points.
In addition, Howe's criterion 
(see \S\ref{S:principal-polarizations}) implies that this isogeny class cannot contain a principally polarized abelian fourfold
over $\FF_{2^g}$ for any odd integer $g$. Consequently, if Conjecture~\ref{conj:Shankar-Tsimerman} holds,
then this isogeny class contains a Jacobian over $\FF_{2^g}$ for some even integer $g \geq 6$.
\end{example}

\subsection{Distinguishing Isogeny Classes by Point Counts}
The Weil polynomial of an abelian variety of dimension $g$ has $g$ unknown coefficients, so it is expected that these can be solved for using $g$ point counts. 
It turns out that we can often do better than this. 
This exotic phenomenon is governed by information-theoretic heuristics (see Remark~\ref{R:information-theory}).
\begin{example}
For $A$ a 5-dimensional abelian variety over $\FF_2$, 
$A$ is determined up to isogeny by the tuple $(\#A(\FF_{2^i}): i = 1,\dots,4)$. This is best possible: for example, an abelian variety $A$ in any of the isogeny classes
\begin{center}
\avlink{5.2.ab\_b\_c\_d\_ac},
\avlink{5.2.ab\_c\_a\_a\_i},
\avlink{5.2.ab\_c\_b\_a\_d}
\end{center}
satisfies $(\#A(\FF_2), \#A(\FF_4), \#A(\FF_8)) = (42, 2520, 80262)$.
By contrast, over $\FF_3$, a 5-dimensional abelian variety is determined up to isogeny by the tuple $(\#A(\FF_{3^i}): i = 1,\dots,3)$.
\end{example}

\begin{example}
Similarly, for $A$ a 6-dimensional abelian variety over $\FF_2$, 
$A$ is determined up to isogeny by the tuple $(\#A(\FF_{2^i}): i = 1,\dots,4)$. This is
best possible even if we restrict to simple abelian varieties: for example, an abelian variety $A$ in any of the isogeny classes
\begin{center}
\avlink{6.2.ab\_ab\_b\_g\_ab\_aj},
\avlink{6.2.ab\_b\_ab\_ac\_f\_ad},
\avlink{6.2.ab\_c\_ab\_ad\_f\_ap}
\end{center}
is geometrically simple and satisfies $(\#A(\FF_2), \#A(\FF_4), \#A(\FF_8)) = (42,4032,246078)$.
\end{example}

\begin{remark}\label{R:information-theory}
It is not known exactly how many initial point counts are needed to identify a $g$-dimensional abelian variety over $\FF_q$ up to isogeny (for known $g$ and $q$). The fact that the complete sequence of point counts determines the Weil polynomial is already nontrivial; it follows from a theorem of Fried
\cite{Fried1988} (see also \cite{Hillar2005}).
Using the Weil bounds, it is shown in \cite{Kedlaya2006b} that
at most $\max\{18,2g\}$ counts suffice; however, Noam Elkies has pointed out that on information-theoretic grounds, one should expect the number of counts needed to be about $g/2$, and indeed this is consistent with these examples. Namely, we need to distinguish among $O(q^{g(g+1)/4})$ Weil polynomials
(\S \ref{subsec:number of isogeny classes}), whereas the Lang-Weil bound 
(\S \ref{ssec:point_count_bounds}) implies that the number of possible values for the tuple $(\#A(\FF_{q^i}))_{i=1}^n$ is $O(q^{n^2/2})$.
\end{remark}

\section{Possible Generalizations and Bottlenecks}
\label{S:future directions}

We conclude with some discussion about possible future directions for this work.

\subsection{Bottlenecks} \label{ssec:bottlenecks}

We first identify some bottleneck steps that limited our original work, and which we would like to overcome.

As described in \S\ref{ssec:angle rank3}, we currently present Frobenius angle ranks which were computed nonrigorously using floating-point arithmetic, because it is not feasible to run the rigorous algorithm on all cases in the database. However, since we also compute Galois groups, we can use those to certify some cases as having maximal angle rank. This covers the vast majority of cases, which might make it feasible to run the rigorous algorithm on the rest, but our present methods rely on the computation of a splitting field, which is also very costly.

As has come up on several occasions already, we do not currently implement positive Jacobian testing for abelian varieties of dimension greater than or equal to 4. See \S\ref{ssec:Jacobians} for further discussion.

One obstruction to adding complete tables of abelian varieties for other pairs $(g,q)$ is the overall size of the dataset.  Using equation \eqref{eq:DiPippoHowe}, including $(7,2)$ would add about $2.2$ million isogeny classes, $(6,3)$ about $10$ million, $(5,4)$ about $2.2$ million, $(4,7)$ about $700000$, and $(3, 27)$ about $450000$.  For comparison, the database currently contains about $3$ million isogeny classes and takes up about $10$ GB.  It would certainly be feasible to extend, but a line needs to be drawn somewhere.  Moreover, the computational time required per isogeny grows quickly with $g$, so adding data with $g \ge 5$ takes more effort than suggested by the number of classes alone. 

In lieu of enlarging the tables, one could also implement the computation of the data presented in LMFDB
on a case-by-case basis for individual isogeny classes. One piece of data that would be difficult to compute
in this way is the isogeny twists, which we currently do by finding hash collisions across the entire table
(\S \ref{ssec:base change}).

It would also be useful to have a mechanism to produce random elements of the set of isogeny classes for large $g$ and $q$. It should be possible to effectively simulate the uniform distribution on isogeny classes by computing analogues of the isogeny Sato-Tate distribution in which one projects onto the first $k$ polynomial coefficients (the isogeny Sato-Tate distribution being the case $k=1$).
An alternative approach may also be to take the existing Weil polynomial iterator discussed in Remark~\ref{rem:iterator} and replace iterations over integers in an interval with uniform random samples.

\subsection{Jacobians} \label{ssec:Jacobians}

Currently the LMFDB does not identify any isogeny class of abelian varieties of dimension at least 4 as containing a Jacobian. For dimension 4, it may be possible to exhaust over isomorphism classes of genus 4 curves using the fact that every such curve is either hyperelliptic or the transverse intersection of a quadric and a cubic surface in $\PP^3$; see \cite{Savitt2003} for a similar calculation.
For curves of genus 5, the analogous assertion is that every such curve is either hyperelliptic, trigonal, or the transverse intersection of three quadrics in $\PP^4$.
For curves of genus 6, the analogous assertion is that every such curve is either hyperelliptic, trigonal,
bielliptic, isomorphic to a smooth plane quintic, or birational to a plane sextic with four double points
 \cite[Exercises~V.A]{Arbarello1985}. (Beware that the previous assertions are made over an algebraically closed field.)

\subsection{Isomorphism Classes} \label{ssec:isormorphism classes}

Beyond the current data in the LMFDB, it would be extremely desirable to tabulate abelian varieties up to isomorphism, not just up to isogeny. 
The easiest cases for this are those of ordinary abelian varieties, for which an explicit description of isomorphism classes within an isogeny class has been given by Deligne \cite{Deligne1969}; abelian varieties over $\FF_p$, for which a similar description has been given by Centeleghe--Stix
\cite{Centeleghe2015}; and almost ordinary abelian varieties \cite{OswalShankar}. The recent work of Marseglia \cite{Marseglia2019} has made great strides towards
making these methods practical at the scale of the LMFDB, and we plan on working with him on including data on isomorphism classes into the database.

In order to handle nonordinary abelian varieties over nonprime fields, it is probably necessary to go
back to the proof of Honda's theorem, by constructing CM abelian varieties over number fields and then reducing them modulo primes. Note that whereas Honda's original approach to this in \cite{Honda1967} used complex uniformization and GAGA, a more recent construction of Chai--Oort \cite{CO15} gives a more algebraic approach that might be easier to implement as an algorithm.

\subsection{K3 Surfaces and Higher Weight}

It would be natural to attempt a similar compilation of other types of algebraic varieties over finite fields and their zeta functions. A strong candidate class for this is K3 surfaces, for which a weak version of the Honda-Tate theorem is known \cites{Taelman2016, Ito2019}.
More precisely, for a given candidate Weil polynomial in this setting, one can prove there exists a corresponding K3 surface in some base change of this isogeny class. 
 The code for tabulating Weil polynomials
described in \S\ref{ssec:enumerating Weil polynomials} can produce lists of possible zeta functons for K3 surfaces over $\FF_q$ for small $q$ (this has been tested up to $q=5$).

This suggests the question of trying to determine exactly which zeta functions occur for K3 surfaces over a given field. An indication of the difficulties involved can be seen in 
\cite{Kedlaya2015}, where a complete tabulation of smooth quartic surfaces in $\PP^3$ over $\FF_2$ and their zeta functions was made (computing the latter by enumerating points);
while this search did realize every eligible zeta function that could not occur for any other type of K3 surface,
not every zeta function that could have appeared did so. Furthermore, certain zeta functions can only appear for K3 surfaces of very large degree, which would be very difficult to write down explicitly (the moduli space of K3 surfaces of a given degree becomes increasingly hyperbolic as the degree increases).

A closely related case is that of cubic fourfolds. Some examples of zeta function computations to resolve specific existence questions for cubic fourfolds can be found in \cite{Addington2018} and \cite{Costa2019}.

In another direction, one can also identify a class of surfaces with small invariants (genus, irregularity, $K^2$); identify all Weil polynomials over $\FF_q$ (for some given small values of $q$) which could arise from a surface with the given invariants; then exhaust over the surfaces in question to see which polynomials arise and what ``isogeny classes'' they fall into. (In general it is not clear what geometric conditions correspond to an equality of Weil polynomials. One nontrivial example is that surfaces which are \emph{derived equivalent}, meaning that they have isomorphic bounded derived categories of coherent sheaves,
have the same zeta function \cite{Honigs2015}.)
In some cases, candidate polynomials can be ruled out because they would predict impossible point counts on the underlying variety; for example, this happens for K3 surfaces over $\FF_2$ as shown in
\cite{Kedlaya2015}.

\appendix

\section{Tables and Figures}

\begin{longtable}{lll}
	Dimension & Num of groups & Num of Newton polygons \\ \hline
	\endfirsthead
	Dimension & Num of groups & Num of Newton polygons \\ \hline
	\endhead
	$1$ & $2$ & $2$ \\ 
	$2$ & $4$ & $3$ \\
	$3$ & $5$ & $5$ \\
	$4$ & $30$ & $8$ \\
	$5$ & $9$ & $12$ \\
	$6$ & $46$ & $20$ \\
	\caption{The number of distinct Galois groups and distinct Newton polygons, by dimension}
\end{longtable}

\LTcapwidth=\textwidth
\begin{longtable}{llllllll}
	Dimension & $G$ & $a(G)$ & $\widetilde{a}_{\maxp}(G)$ & $q=2$ & $q=3$ & $q=4$ & $q=5$ \\ \hline
	\endfirsthead
	Dimension & $G$ & $a(G)$ & $\widetilde{a}_{\maxp}(G)$ & $q=2$ & $q=3$ & $q=4$ & $q=5$ \\ \hline
	\endhead
	$1$ & $\text{\texttt{1T1}}$ &  & $-\infty$ &  &  & $2$ &  \\
	$1$ & $\text{\texttt{2T1}}$ & 1 & 1.000 & $5$ & $7$ & $7$ & $9$ \\
	$2$ & $\text{\texttt{2T1}}$ & 1 & 0.000 & $1$ & $1$ &  & $1$ \\
	$2$ & $\text{\texttt{4T1}}$ & 1/2 & 0.370 &  & $2$ & $2$ & $2$ \\
	$2$ & $\text{\texttt{4T2}}$ & 1/2 & 0.661 & $11$ & $16$ & $14$ & $29$ \\
	$2$ & $\text{\texttt{4T3}}$ & 1 & 0.997 & $8$ & $16$ & $30$ & $52$ \\
	$3$ & $\text{\texttt{6T1}}$ & 1/3 & 0.199 & $4$ & $2$ & $6$ &  \\
	$3$ & $\text{\texttt{6T3}}$ & 1/2 & 0.557 & $14$ & $32$ & $84$ & $88$ \\
\end{longtable}

\begin{longtable}{llllll}
	Dimension & Slopes & $q=2$ & $q=3$ & $q=4$ & $q=5$ \\ \hline
	\endfirsthead
	Dimension & Slopes & $q=2$ & $q=3$ & $q=4$ & $q=5$ \\ \hline
	\endhead
	$1$ & $\left(0, 1\right)$ & $2$ & $4$ & $4$ & $8$ \\
	$1$ & $\left(\frac{1}{2}, \frac{1}{2}\right)$ & $3$ & $3$ & $5$ & $1$ \\
	$2$ & $\left(0, 0, 1, 1\right)$ & $13$ & $30$ & $34$ & $66$ \\
	$2$ & $\left(0, \frac{1}{2}, \frac{1}{2}, 1\right)$ & $2$ & $2$ & $8$ & $12$ \\
	$2$ & $\left(\frac{1}{2}, \frac{1}{2}, \frac{1}{2}, \frac{1}{2}\right)$ & $5$ & $3$ & $4$ & $6$ \\
	$3$ & $\left(0, 0, 0, 1, 1, 1\right)$ & $56$ & $266$ & $576$ & $1696$ \\
	$3$ & $\left(0, 0, \frac{1}{2}, \frac{1}{2}, 1, 1\right)$ & $12$ & $54$ & $192$ & $284$ \\
	$3$ & $\left(0, \frac{1}{2}, \frac{1}{2}, \frac{1}{2}, \frac{1}{2}, 1\right)$ & $4$ & $16$ & $30$ & $36$ \\
\end{longtable}

\begin{longtable}{lllllll}
	Slopes & Group & Angle Rank & $q=2$ & $q=3$ & $q=4$ & $q=5$ \\ \hline
	\endfirsthead
	Slopes & Group & Angle Rank & $q=2$ & $q=3$ & $q=4$ & $q=5$ \\ \hline
	\endhead
	$\left(0, 1\right)$ & $\text{\texttt{2T1}}$ & $1$ & $2$ & $4$ & $4$ & $8$ \\
	$\left(\frac{1}{2}, \frac{1}{2}\right)$ & $\text{\texttt{1T1}}$ & $0$ &  &  & $2$ &  \\
	$\left(\frac{1}{2}, \frac{1}{2}\right)$ & $\text{\texttt{2T1}}$ & $0$ & $3$ & $3$ & $3$ & $1$ \\
	\caption{Counts of isogeny classes by Newton polygon, Galois group, and angle rank ($g=1$).} \\
\end{longtable}

\vspace{0.2in}

\begin{longtable}{lllllll}
	Slopes & Group & Angle Rank & $q=2$ & $q=3$ & $q=4$ & $q=5$ \\ \hline
	\endfirsthead
	Slopes & Group & Angle Rank & $q=2$ & $q=3$ & $q=4$ & $q=5$ \\ \hline
	\endhead
	$\left(0, 0, 1, 1\right)$ & $\text{\texttt{4T1}}$ & $2$ &  & $2$ &  &  \\
	$\left(0, 0, 1, 1\right)$ & $\text{\texttt{4T2}}$ & $1$ & $7$ & $14$ & $12$ & $26$ \\
	$\left(0, 0, 1, 1\right)$ & $\text{\texttt{4T3}}$ & $2$ & $6$ & $14$ & $22$ & $40$ \\
	$\left(0, \frac{1}{2}, \frac{1}{2}, 1\right)$ & $\text{\texttt{4T3}}$ & $2$ & $2$ & $2$ & $8$ & $12$ \\
	$\left(\frac{1}{2}, \frac{1}{2}, \frac{1}{2}, \frac{1}{2}\right)$ & $\text{\texttt{2T1}}$ & $0$ & $1$ & $1$ &  & $1$ \\
	$\left(\frac{1}{2}, \frac{1}{2}, \frac{1}{2}, \frac{1}{2}\right)$ & $\text{\texttt{4T1}}$ & $0$ &  &  & $2$ & $2$ \\
	$\left(\frac{1}{2}, \frac{1}{2}, \frac{1}{2}, \frac{1}{2}\right)$ & $\text{\texttt{4T2}}$ & $0$ & $4$ & $2$ & $2$ & $3$ \\
	\caption{Counts of isogeny classes by Newton polygon, Galois group, and angle rank ($g=2$).}
\end{longtable}

\vspace{0.2in}

\begin{longtable}{lllllll}
	Slopes & Group & Angle Rank & $q=2$ & $q=3$ & $q=4$ & $q=5$ \\ \hline
	\endfirsthead
	Slopes & Group & Angle Rank & $q=2$ & $q=3$ & $q=4$ & $q=5$ \\ \hline
	\endhead
	$\left(0, 0, 0, 1, 1, 1\right)$ & $\text{\texttt{6T1}}$ & $1$ & $4$ &  & $4$ &  \\
	$\left(0, 0, 0, 1, 1, 1\right)$ & $\text{\texttt{6T3}}$ & $1$ & $4$ & $10$ & $12$ & $28$ \\
	$\left(0, 0, 0, 1, 1, 1\right)$ & $\text{\texttt{6T3}}$ & $3$ & $2$ & $6$ & $10$ & $42$ \\
	$\left(0, 0, 0, 1, 1, 1\right)$ & $\text{\texttt{6T6}}$ & $3$ & $14$ & $32$ & $48$ & $82$ \\
	$\left(0, 0, 0, 1, 1, 1\right)$ & $\text{\texttt{6T11}}$ & $3$ & $32$ & $218$ & $502$ & $1544$ \\
	$\left(0, 0, \frac{1}{2}, \frac{1}{2}, 1, 1\right)$ & $\text{\texttt{6T3}}$ & $2$ & $6$ & $12$ & $56$ & $6$ \\
	$\left(0, 0, \frac{1}{2}, \frac{1}{2}, 1, 1\right)$ & $\text{\texttt{6T6}}$ & $3$ &  &  & $2$ & $10$ \\
	$\left(0, 0, \frac{1}{2}, \frac{1}{2}, 1, 1\right)$ & $\text{\texttt{6T11}}$ & $3$ & $6$ & $42$ & $134$ & $268$ \\
\end{longtable}

\begin{longtable}{lllllll}
	Slopes & Group & Angle Rank & $q=2$ & $q=3$ & $q=4$ & $q=5$ \\ \hline
	\endfirsthead
	Slopes & Group & Angle Rank & $q=2$ & $q=3$ & $q=4$ & $q=5$ \\ \hline
	\endhead
	$\left(0, 0, 0, 0, 1, 1, 1, 1\right)$ & $\text{\texttt{8T2}}$ & $1$ & $8$ & $12$ & $10$ & $10$ \\
	$\left(0, 0, 0, 0, 1, 1, 1, 1\right)$ & $\text{\texttt{8T2}}$ & $2$ &  & $3$ &  &  \\
	$\left(0, 0, 0, 0, 1, 1, 1, 1\right)$ & $\text{\texttt{8T3}}$ & $1$ & $4$ & $12$ & $9$ & $19$ \\
	$\left(0, 0, 0, 0, 1, 1, 1, 1\right)$ & $\text{\texttt{8T4}}$ & $2$ &  &  & $1$ & $3$ \\
	$\left(0, 0, 0, 0, 1, 1, 1, 1\right)$ & $\text{\texttt{8T6}}$ & $4$ &  & $8$ &  & $8$ \\
	$\left(0, 0, 0, 0, 1, 1, 1, 1\right)$ & $\text{\texttt{8T9}}$ & $1$ & $6$ & $20$ & $28$ & $80$ \\
	$\left(0, 0, 0, 0, 1, 1, 1, 1\right)$ & $\text{\texttt{8T9}}$ & $2$ & $19$ & $59$ & $100$ & $221$ \\
	$\left(0, 0, 0, 0, 1, 1, 1, 1\right)$ & $\text{\texttt{8T9}}$ & $4$ &  &  &  & $8$ \\
\end{longtable}

\begin{longtable}{lllll}
	Slopes & Group & Angle Rank & $q=2$ & $q=3$ \\ \hline
	\endfirsthead
	Slopes & Group & Angle Rank & $q=2$ & $q=3$ \\ \hline
	\endhead
	$\left(0, 0, 0, 0, 0, 1, 1, 1, 1, 1\right)$ & $\text{\texttt{10T1}}$ & $1$ &  & $4$ \\
	$\left(0, 0, 0, 0, 0, 1, 1, 1, 1, 1\right)$ & $\text{\texttt{10T5}}$ & $1$ & $10$ & $38$ \\
	$\left(0, 0, 0, 0, 0, 1, 1, 1, 1, 1\right)$ & $\text{\texttt{10T5}}$ & $5$ &  & $2$ \\
	$\left(0, 0, 0, 0, 0, 1, 1, 1, 1, 1\right)$ & $\text{\texttt{10T14}}$ & $5$ & $8$ & $26$ \\
	$\left(0, 0, 0, 0, 0, 1, 1, 1, 1, 1\right)$ & $\text{\texttt{10T22}}$ & $5$ & $34$ & $292$ \\
	$\left(0, 0, 0, 0, 0, 1, 1, 1, 1, 1\right)$ & $\text{\texttt{10T23}}$ & $5$ & $12$ & $80$ \\
	$\left(0, 0, 0, 0, 0, 1, 1, 1, 1, 1\right)$ & $\text{\texttt{10T29}}$ & $5$ & $2$ & $50$ \\
	$\left(0, 0, 0, 0, 0, 1, 1, 1, 1, 1\right)$ & $\text{\texttt{10T36}}$ & $5$ & $2$ & $52$ \\
\end{longtable}

\begin{longtable}{llll}
	Slopes & Group & Angle Rank & $q=2$ \\ \hline
	\endfirsthead
	Slopes & Group & Angle Rank & $q=2$ \\ \hline
	\endhead
	$\left(0, 0, 0, 0, 0, 0, 1, 1, 1, 1, 1, 1\right)$ & $\text{\texttt{12T2}}$ & $1$ & $10$ \\
	$\left(0, 0, 0, 0, 0, 0, 1, 1, 1, 1, 1, 1\right)$ & $\text{\texttt{12T3}}$ & $1$ & $1$ \\
	$\left(0, 0, 0, 0, 0, 0, 1, 1, 1, 1, 1, 1\right)$ & $\text{\texttt{12T10}}$ & $1$ & $9$ \\
	$\left(0, 0, 0, 0, 0, 0, 1, 1, 1, 1, 1, 1\right)$ & $\text{\texttt{12T10}}$ & $3$ & $3$ \\
	$\left(0, 0, 0, 0, 0, 0, 1, 1, 1, 1, 1, 1\right)$ & $\text{\texttt{12T18}}$ & $1$ & $6$ \\
	$\left(0, 0, 0, 0, 0, 0, 1, 1, 1, 1, 1, 1\right)$ & $\text{\texttt{12T21}}$ & $3$ & $1$ \\
	$\left(0, 0, 0, 0, 0, 0, 1, 1, 1, 1, 1, 1\right)$ & $\text{\texttt{12T23}}$ & $3$ & $3$ \\
	$\left(0, 0, 0, 0, 0, 0, 1, 1, 1, 1, 1, 1\right)$ & $\text{\texttt{12T23}}$ & $5$ & $6$ \\
\end{longtable}

\begin{longtable}{llllll}
	$g$ & $\End^0(A/\Fq)$ invs & $\End^0(A/\overline{\FF}_q)$ invs & Count & $q$s & Example \\ \hline
	\endfirsthead
	$g$ & $\End^0(A/\Fq)$ invs & $\End^0(A/\overline{\FF}_q)$ invs & Count & $q$s & Example \\ \hline
	\endhead
	$1$ & zeroes & zeroes & $6000$ & all & \avlink{1.2.ab} \\
	    & $(0)$ & $(1/2)$ & $154$ & all except $169$ & \avlink{1.2.ac} \\
	    & $(1/2)$ & $(1/2)$ & $30$ & squares & \avlink{1.4.e} \\
	$2$ & zeroes & zeroes & $1198020$ & all & \avlink{2.2.b\_a} \\
	    & $(0)$ & $(1/2)$ & $195$ & all except $121$ & \avlink{2.2.a\_ae} \\
	    & $(0, 0)$ & $(1/2)$ & $80$ & $5, 7, 9, 13, 17, 25\ldots$ & \avlink{2.5.a\_a} \\
	    & $(1/2, 1/2)$ & $(1/2)$ & $7$ & $25, 49, 169, 625$ & \avlink{2.25.a\_by} \\
	$3$ & zeroes & zeroes & $962658$ & all & \avlink{3.2.b\_b\_d} \\
\end{longtable}

\begin{figure}[!htbp]
\begin{center}
	\includegraphics[scale=0.4]{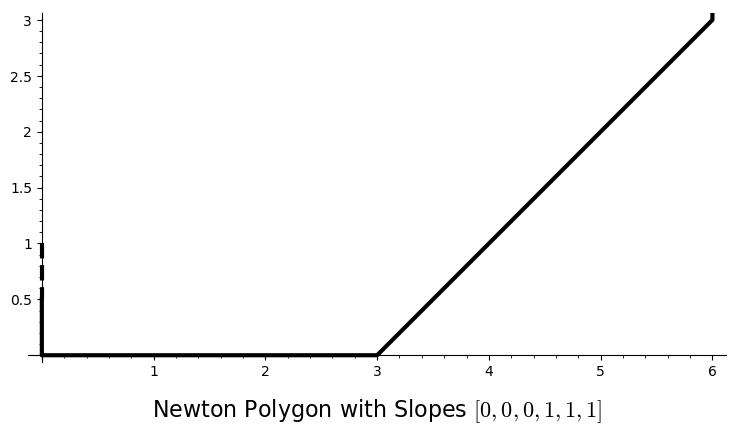}  
	\includegraphics[scale=0.4]{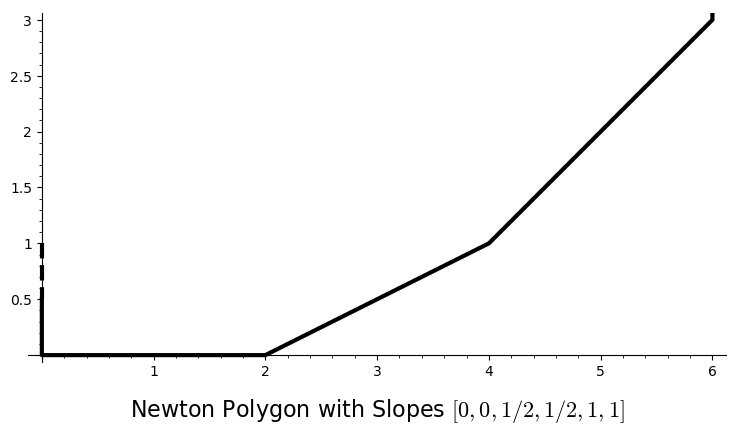}
\end{center}
\begin{center}
	\includegraphics[scale=0.4]{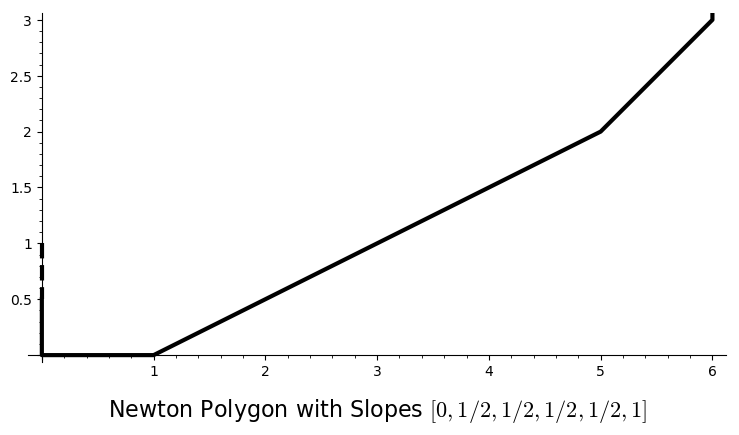}
	\includegraphics[scale=0.4]{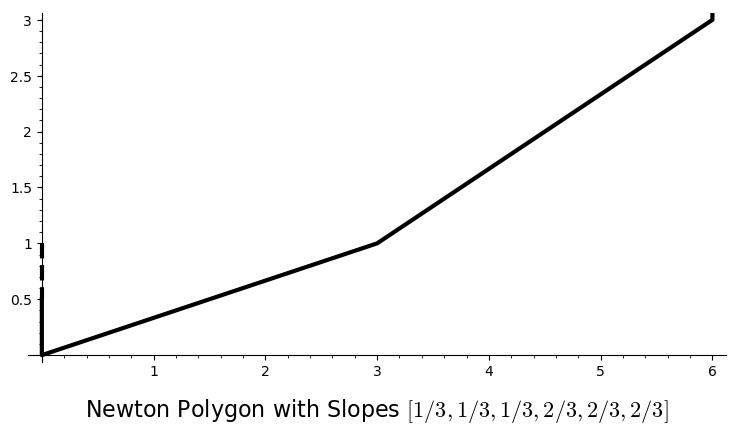}  
\end{center}
\begin{center}
	\includegraphics[scale=0.4]{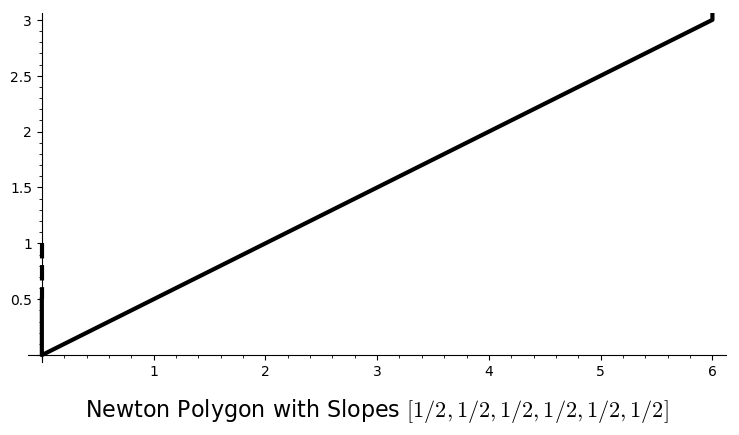}  
\end{center}
\caption{Possible Newton polygons for dimension 3 abelian varieties.
The partial ordering is linear. \label{F:newton}} 
\end{figure}

\begin{figure}[!htbp]
\begin{center}
\begin{tabular}{cc}
\includegraphics[scale=0.5]{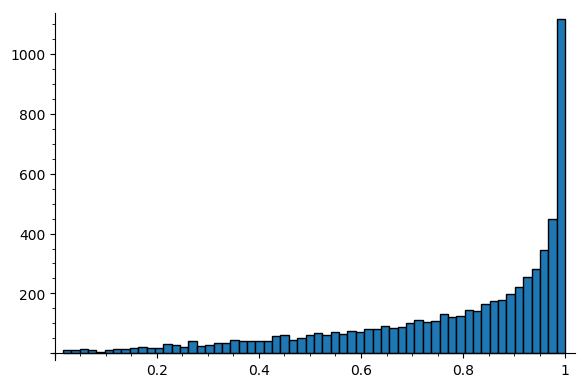} &
\includegraphics[scale=0.5]{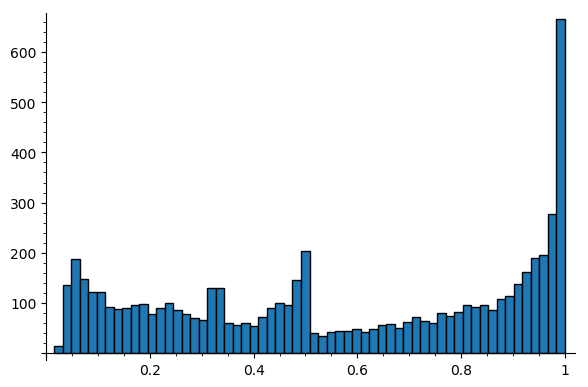} \\
Polynomial & Number field
\end{tabular}
\end{center}
\caption{Normalized root discriminants for $g=1$.\label{F:disc-hist-1} }
\end{figure}

\begin{figure}[!htbp]
\begin{center}	
\begin{tabular}{cc}
\includegraphics[scale=0.5]{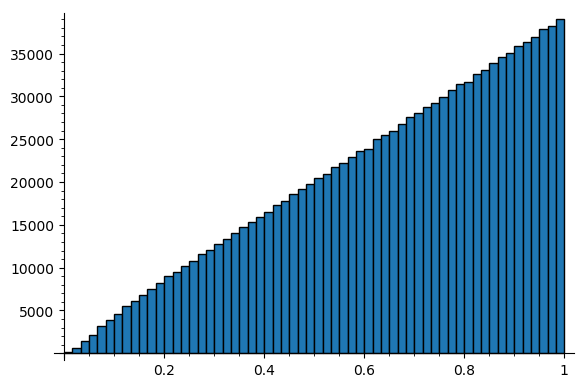} &
\includegraphics[scale=0.5]{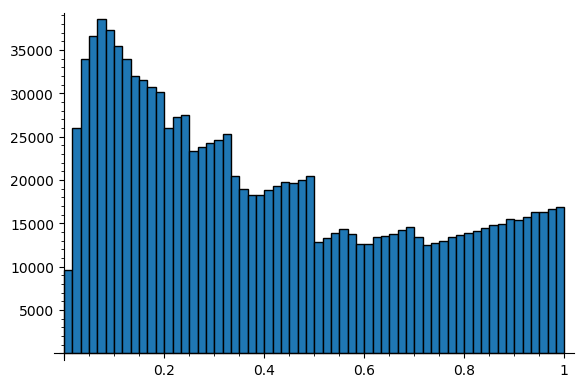} \\
Polynomial & Number field
\end{tabular}
\end{center}
\caption{Normalized root discriminants for $g=2$.\label{F:disc-hist-2} }
\end{figure}

\begin{figure}[h]
\begin{center}	
\begin{tabular}{cc}
\includegraphics[scale=0.5]{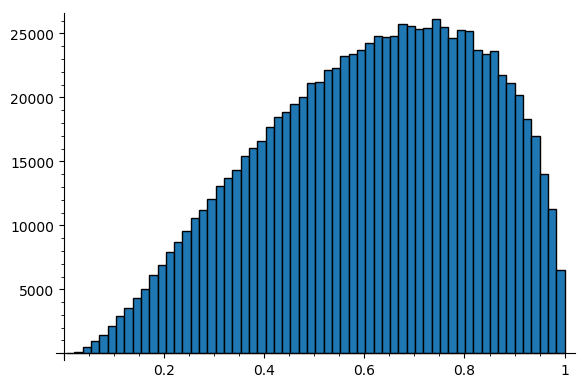} &
\includegraphics[scale=0.5]{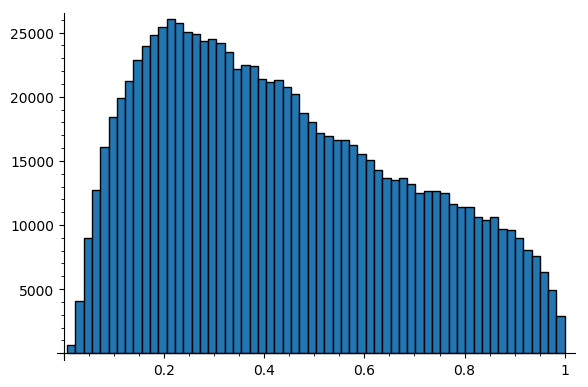} \\
Polynomial & Number field
\end{tabular}
\end{center}
\caption{Normalized root discriminants for $g=3$.\label{F:disc-hist-3} }
\end{figure}

\begin{figure}[!htbp]
\begin{center}	
\begin{tabular}{cc}
\includegraphics[scale=0.5]{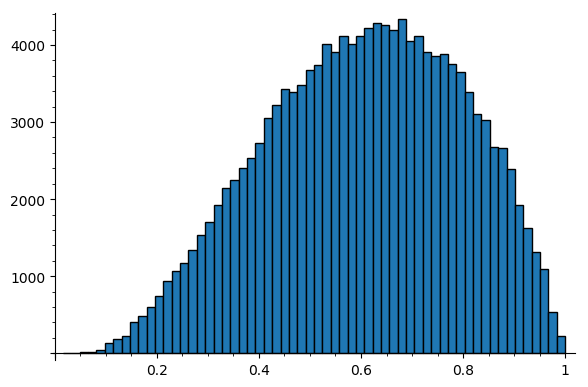} &
\includegraphics[scale=0.5]{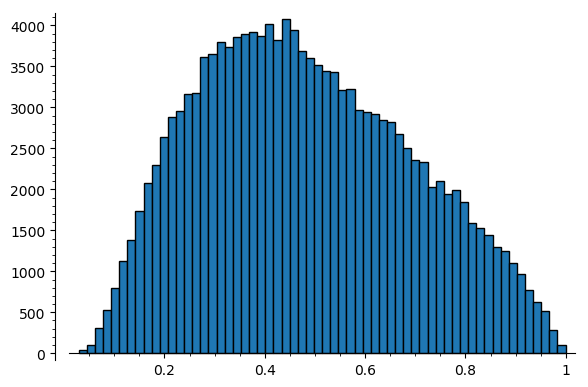} \\
Polynomial & Number field
\end{tabular}
\end{center}
\caption{Normalized root discriminants for $g=4$.\label{F:disc-hist-4} }
\end{figure}

\begin{figure}[!htbp]
\begin{center}	
\begin{tabular}{cc}
\includegraphics[scale=0.5]{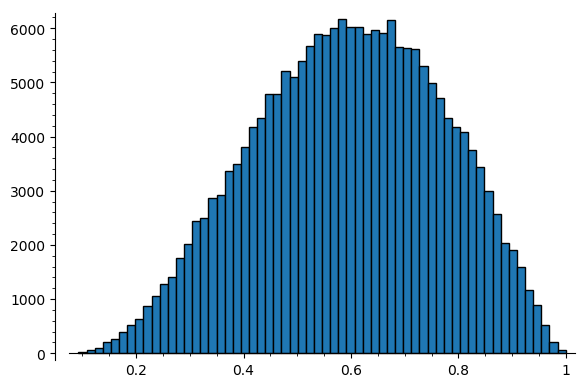} &
\includegraphics[scale=0.5]{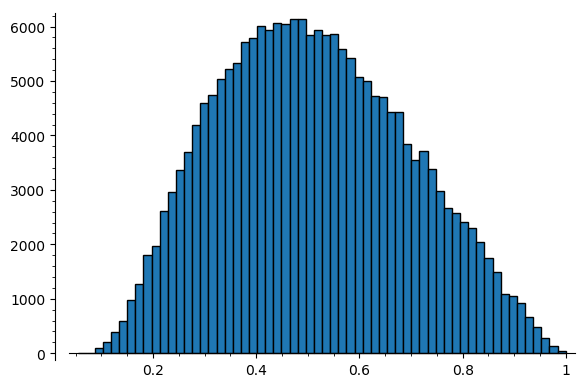} \\
Polynomial & Number field
\end{tabular}
\end{center}
\caption{Normalized root discriminants for $g=5$.\label{F:disc-hist-5} }
\end{figure}

\begin{figure}[!htbp]
\begin{center}	
\begin{tabular}{cc}
\includegraphics[scale=0.5]{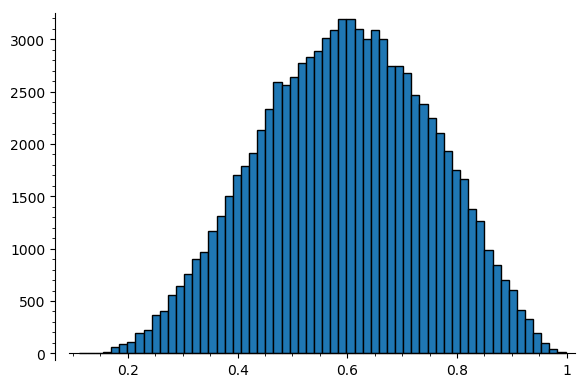} &
\includegraphics[scale=0.5]{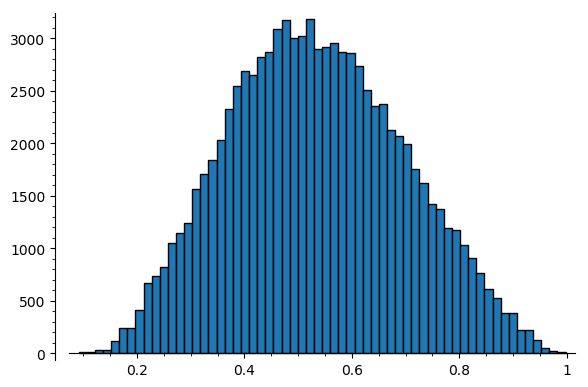} \\
Polynomial & Number field
\end{tabular}
\end{center}
\caption{Normalized root discriminants for $g=6$.\label{F:disc-hist-6} }
\end{figure}

\begin{figure}[!htbp]
\begin{center}
	\includegraphics[scale=0.35]{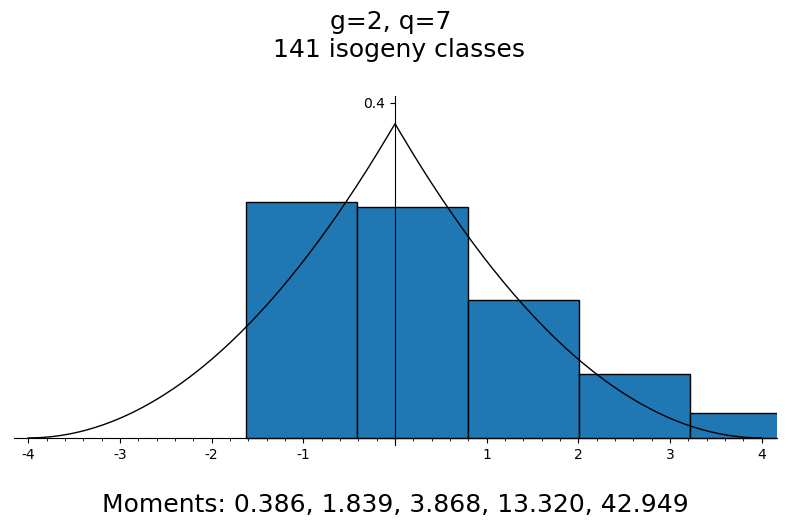}
	\includegraphics[scale=0.35]{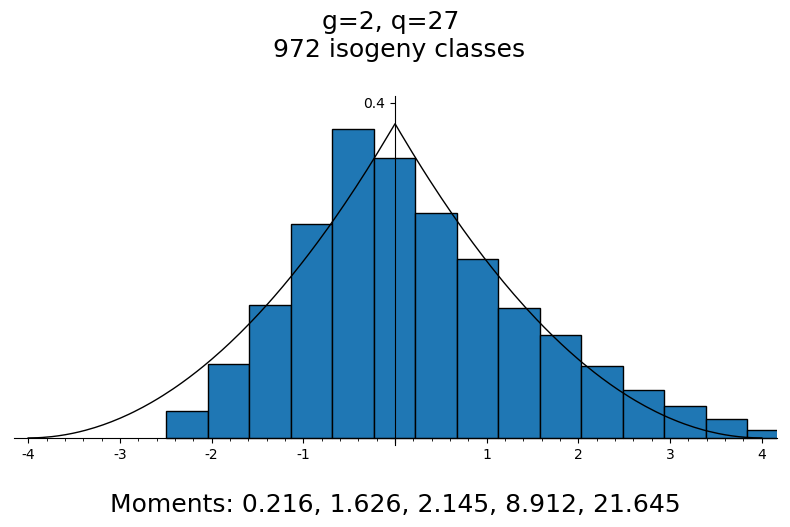}
\end{center}
\begin{center}
	\includegraphics[scale=0.35]{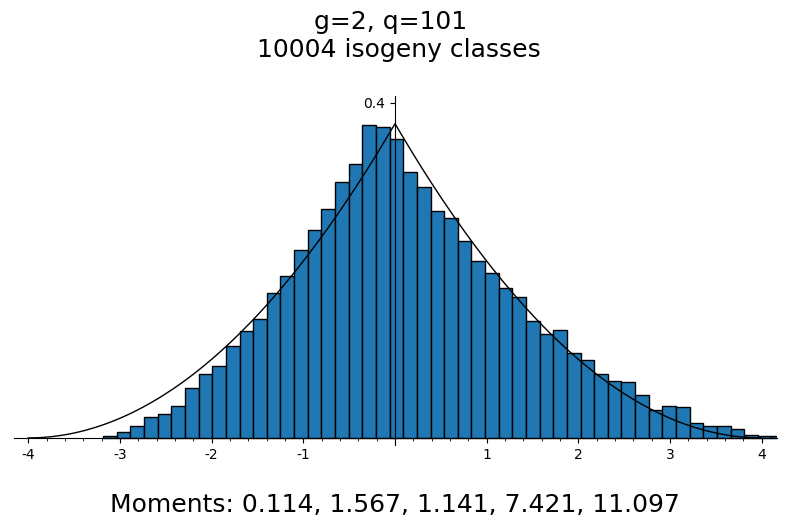}
	\includegraphics[scale=0.35]{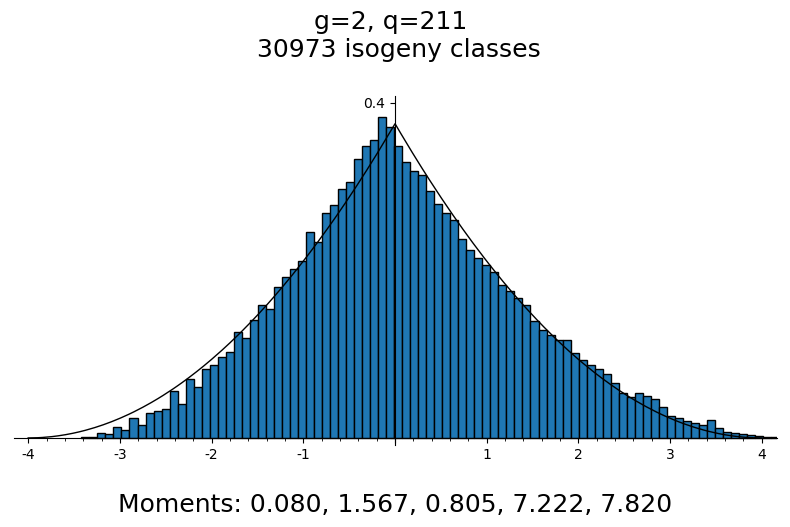}
\end{center}
\caption{Isogeny Sato-Tate distribution for $g=2$ with $q=7,27,101,211$.\label{F:notto2}}
\end{figure}

\begin{figure}[!htbp]
	\begin{center}
		\includegraphics[scale=0.35]{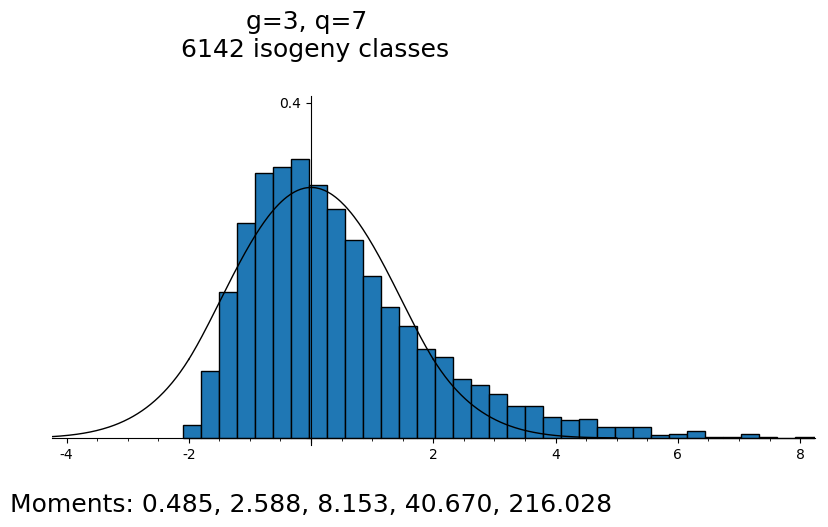}
		\includegraphics[scale=0.35]{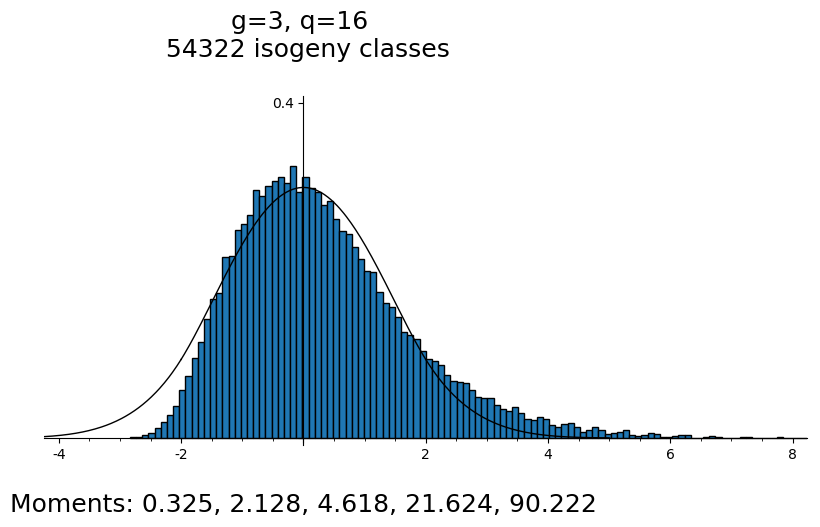}
	\end{center}
	\begin{center}
		\includegraphics[scale=0.35]{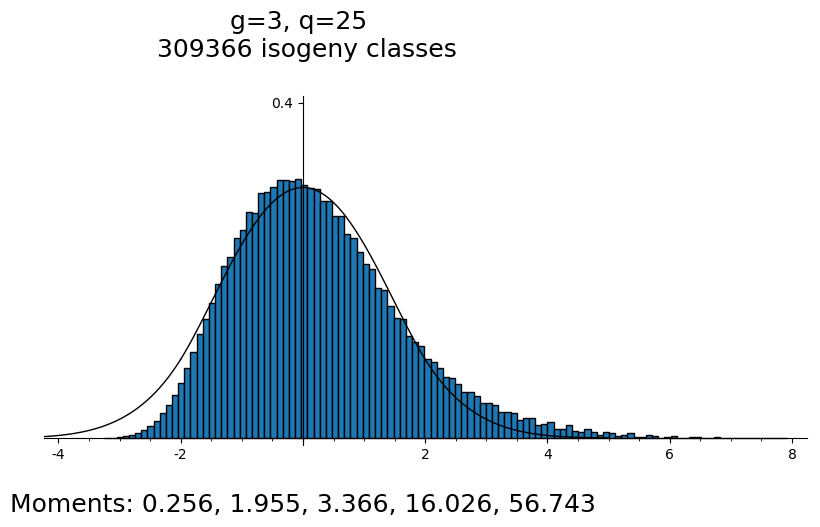}
	\end{center}
	\caption{Isogeny Sato-Tate distribution for $g=3$ with $q=7,16$ and $25$.\label{F:notto3}}
\end{figure}

\begin{figure}[!htbp]
    \begin{center}
		\includegraphics[scale=0.5]{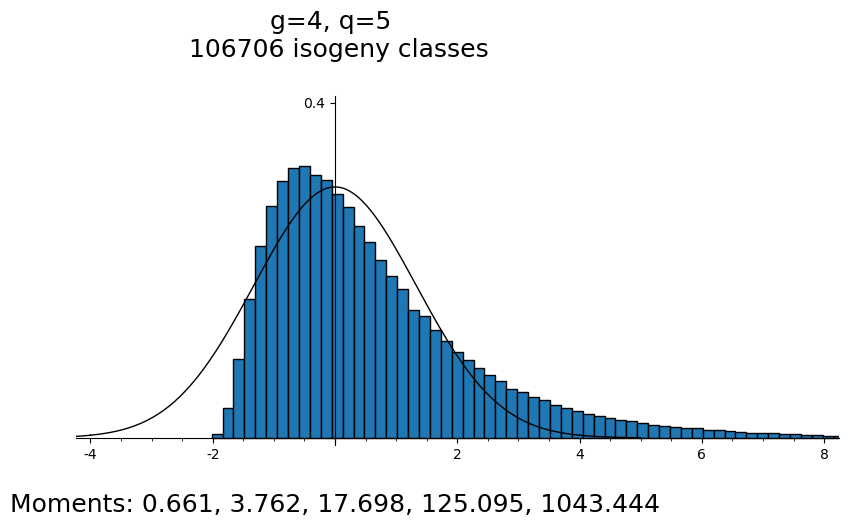}
	\end{center}
	\caption{Isogeny Sato-Tate distribution for $g=4$ and $q=5$.\label{F:notto4}}
\end{figure}

\clearpage

\bibliographystyle{alpha}
\bibliography{weil}
\end{document}